\documentclass[11pt, reqno]{amsart}
\usepackage{latexsym,amsmath,amssymb, amsthm, mathscinet}
\usepackage{cases, verbatim}
\usepackage[active]{srcltx}
\usepackage{hyperref}

\usepackage{xcolor}

\usepackage[margin=1.5in]{geometry}

\newtheorem{thm}{Theorem}
\newtheorem{lem}[thm]{Lemma}
\newtheorem{prop}[thm]{Proposition}

\theoremstyle{remark}
\newtheorem{rmk}[thm]{Remark}
\newtheorem{example}[thm]{Example}

\numberwithin{equation}{section}

\theoremstyle{definition}
\newtheorem{defi}[thm]{Definition}

\numberwithin{thm}{section} 
\numberwithin{equation}{section}

\makeatletter

\newcommand{\Rmnum}[1]{\expandafter\@slowromancap\romannumeral #1@}
\makeatother

\def\R{{\mathbb R}}

\def\N{{\mathcal N}}
\def\H{{\mathbb H}}

\def\A{{\mathcal A}}

\newcommand{\pO}{\partial\Omega}
\newcommand{\Oba}{\overline{\Omega}}

\newcommand{\vep}{\varepsilon}

\newcommand{\ul}{\underline}

\newcommand{\bpm}{\begin{pmatrix}}
\newcommand{\epm}{\end{pmatrix}}
\newcommand{\la}{\left\langle}
\newcommand{\ra}{\right\rangle}

\newcommand{\beq}{\begin{equation}}
\newcommand{\eeq}{\end{equation}}
\newcommand{\coh}{\operatorname{co}}
\newcommand{\spa}{\operatorname{span}}

\DeclareMathOperator*{\limsups}{limsup^\ast}

plus 6pt minus 12pt
\title[H-quasiconvex envelope]{\protect{Horizontally quasiconvex envelope\\ in the Heisenberg group}}
\author[A. Kijowski]{Antoni Kijowski}
\address[Antoni Kijowski]{Analysis on Metric Spaces Unit, Okinawa Institute of Science and Technology Graduate University, Okinawa 904-0495, Japan, {\tt antoni.kijowski@oist.jp}}

\author[Q. Liu]{Qing Liu}
\address[Qing Liu]{Geometric Partial Differential Equations Unit, Okinawa Institute of Science and Technology Graduate University, Okinawa 904-0495, Japan, {\tt qing.liu@oist.jp}}

\author[X. Zhou]{Xiaodan Zhou}
\address[Xiaodan Zhou]{Analysis on Metric Spaces Unit, Okinawa Institute of Science and Technology Graduate University, Okinawa 904-0495, Japan, {\tt xiaodan.zhou@oist.jp}}

\begin{document}

\begin{abstract}
This paper is concerned with a PDE-based approach to the horizontally quasiconvex (h-quasiconvex for short) envelope of a given continuous function in the Heisenberg group. We provide a characterization for upper semicontinuous, h-quasiconvex functions in terms of the viscosity subsolution to a first-order nonlocal Hamilton-Jacobi equation. 
We also construct the corresponding envelope of a continuous function by iterating the nonlocal operator. 
One important step in our arguments is to prove the uniqueness and existence of viscosity solutions to the Dirichlet boundary problem for the nonlocal Hamilton-Jacobi equation.  Applications of our approach to the h-convex hull of a given set in the Heisenberg group are discussed as well. 
\end{abstract}


\maketitle

\section{Introduction}\label{sec:intro}

\subsection{Background and motivation}
Convex analysis is a classical and fundamental topic with numerous applications in various fields of mathematics and beyond. In contrast to the extensive literature on convex analysis in the Euclidean space, less is known about the case in a general geometric setting such as sub-Riemannian manifolds. This paper is mainly concerned with a PDE method to deal with a certain weak type of convexity for both sets and functions in the first Heisenberg group $\H$. 

The Heisenberg group $\mathbb{H}$ is $\mathbb{R}^{3}$ endowed with the non-commutative group multiplication 
\[
(x_p, y_p, z_p)\cdot (x_q, y_q, z_q)=\left(x_p+x_q, y_p+y_q, z_p+z_q+\frac{1}{2}(x_py_q-x_qy_p)\right),
\]
for all $p=(x_p, y_p, z_p)$ and $q=(x_q, y_q, z_q)$ in $\H$. The differential structure of $\mathbb{H}$ is determined by the left-invariant vector fields
\[
X_1=\frac{\partial}{\partial x}-\frac{y}{2}\frac{\partial}{\partial z}, \quad  X_2=\frac{\partial}{\partial y}+\frac{x}{2}\frac{\partial}{\partial z}, \quad X_3=\frac{\partial}{\partial z}.
\]
  One may easily verify the commutation relation $X_3=[X_1, X_2]=
  X_{1}X_{2}- X_{2}X_{1}$. Let
\[
\mathbb{H}_0=\{h\in \mathbb{H}: h=(x, y, 0) \ \text{ for $x, y\in \mathbb{R}$}\}.
\]
For any $p\in \H$, the set
\[
\H_p=\{p\cdot h: \ h\in \H_0\}
\]
is called the horizontal plane through $p$. It is clear that $\H_p=\spa\{X_1(p), X_2(p)\}$ for every $p\in \H$. See \cite{CDPT} for a detailed introduction of the Heisenberg group.

Our primary interest is to understand how to convexify a given bounded set in the Heisenberg group $\H$, that is, we aim to find the smallest convex set that contains the given set. Let us first clarify the meaning of convex sets we consider in this work. We shall actually focus on the notion of weakly h-convex sets, which is first introduced in \cite{DGN1} and later studied in \cite{Rithesis, CCP1, ArCaMo} etc. A set $E\subset \H$ is said to be weakly h-convex if the horizontal segment connecting any two points in $E$ lies in $E$; see also Definition \ref{def h-set}. Hereafter we call such a set an h-convex set for simplicity of terminology. 

There are several other types of set convexity in $\H$ defined with different kinds of convex combination of two points. One seemingly natural notion is based on geodesics in $\H$.
A set $E\subset \H$ is said to be geodetically convex if, for every pair of points $p, q\in E$, the set $E$ contains all geodesics joining $p$ and $q$. 
This notion is known to be a very strong one; the geodetically convex hull of any three points that are not on the same geodesic in $\H$ has to be the whole group \cite{MoR}.  A different notion, called strong h-convexity \cite{DGN1} or twisted convexity \cite{CCP1}, uses the dilation of group quotient to combine two points. It is still a quite strong notion, much stronger than the Euclidean convexity. In general a strongly h-convex hull of a bounded set consisting of more than two points could be unbounded \cite{CCP1}.  
We refer the reader to \cite{Rithesis, CCP1} for related discussions on these convexity notions. 

The notion of (weak) h-convexity is obviously weaker than the Euclidean convexity as well as the other aforementioned notions because of the restriction on horizontal segments. Such relaxation gives rise to unexpected properties. An h-convex set can even be disconnected, as pointed out in \cite{CCP2}. One simple example of h-convex sets is the union of two points $(0, 0, 1)$ and $(0, 0, -1)$ in $\H$. It is h-convex, because no horizontal segments exist to connect the points. 
A less trivial example of disconnected h-convex sets is presented in Example \ref{ex1}. 
Such an unusual character makes it challenging to find h-convex hull of a given set in $\H$. In contrast to the Euclidean situation, where one can generate the convex hull of a set by simply connecting every pair of points in the set with a segment, building the h-convex hull of a given set in the Heisenberg group requires possibly infinite iterations to include all necessary horizontal segments; see the proof of \cite[Lemma 4.1]{Riconv} by Rickly about this method. 
In general, it is not straightforward to describe and construct the h-convex hull of a set.  This motivates us to search possible analytic methods to solve this problem.

A closely related problem we also intend to discuss is constructing the horizontally quasiconvex (or simply h-quasiconvex) envelope of a given function in an h-convex domain $\Omega \subset \H$. An h-quasiconvex function in $\Omega$ is defined to be a function whose sublevel sets are all h-convex. It is equivalent to saying that 
\beq\label{h-quasi def}
u(w) \leq \max\{u(p),u(q)\}
\eeq
holds for any $p\in \Omega$, $q \in \H_p\cap \Omega$ and $w \in [p,q]$.
This notion is introduced in \cite{SuYa}. It is also studied in \cite{CCP2}, where such functions are called weakly h-quasiconvex functions. We again suppress the term ``weakly'' in this paper to avoid redundancy. Similar to the case of sets, for any function in $\Omega$, it is not trivial in general how one can find its h-quasiconvex envelope in a constructive way. 

We remark that a stronger notion of function convexity in $\H$, called horizontal convexity, is introduced by \cite{DGN1} and \cite{LMS}. Various properties and generalizations of such convex functions are discussed in \cite{BaRi, GuMo2, Wa, Mag, CaPi, BaDr, MaSc} etc. The corresponding convex envelope and its applications to convexity properties of sub-elliptic equations are studied in \cite{LZ2}.

\subsection{Main results}
Inspired by the Euclidean results in \cite{BGJ1}, in this work we provide a PDE-based approach to investigate h-convex hulls and h-quasiconvex envelopes. Our study starts from an improved characterization of h-quasiconvex functions. It is known \cite[Theorem 4.5]{CCP2} that any function $u\in C^1(\Omega)$ is h-quasiconvex if and only if 
\begin{equation}\label{USC}
 u(\xi) < u(p) \Rightarrow \langle \nabla_H u(p), (p^{-1}\cdot \xi)_h \rangle \leq 0
\end{equation}
for any $\xi \in \H_p\cap \Omega$. Here, $\nabla_H u$ denotes the horizontal gradient of $u$, given by $\nabla_H u=(X_1 u, X_2u)$. Also, for each $\zeta=(x_\zeta, y_\zeta, z_\zeta)\in \H$, $\zeta_h$ represents its horizontal coordinates, that is, $\zeta_h=(x_\zeta, y_\zeta)$. For the sake of our applications to construct h-convex hulls we generalize this characterization to fuctions which are not necessarily of $C^1$ class. Extending the Euclidean arguments in \cite{BGJ1} to $\H$, we show in Theorem \ref{thm char} that an upper semicontinuous (USC) function $u$ is h-quasiconvex if and only if \eqref{USC} holds in the viscosity sense. It is equivalent to saying that 
\beq\label{vis char}
\sup\{\la \nabla_H \varphi(p), (p^{-1}\cdot \xi)_h\ra: \xi\in \H_p\cap \Omega, \ u(\xi)< u(p)\}\leq 0
\eeq
whenever there exist $p\in \Omega$ and $\varphi\in C^1(\Omega)$ such that $u-\varphi$ attains a maximum at $p$.

Further developing the generalized characterization, we adopt an iterative scheme to find the h-quasiconvex envelope of a continuous function $f$, denoted by $Q(f)$, in a bounded h-convex domain $\Omega$ under a particular Dirichlet boundary condition. The iteration is implemented by solving a sequence of nonlocal Hamilton-Jacobi equations, where the Hamiltonian is given by the left hand side of \eqref{vis char}, that is, 
\[
H(p, u(p), \nabla u(p))=\sup\{\la \nabla_H u(p), (p^{-1}\cdot \xi)_h\ra: \xi\in \H_p\cap \Omega, \ u(\xi)< u(p)\}.
\]
We briefly describe our scheme in what follows. 

Let $f\in C(\Oba)$ be a given function satisfying
\begin{numcases}{}
f=K \quad &\text{on $\pO$,}\label{dirichlet-f} \\
f\leq K \quad &\text{in $\Oba$}\label{coercive-f}
\end{numcases}
with $K\in \R$. This set of conditions resembles the coercivity assumption on $f$ when $K>0$ is taken large. It guarantees the existence of an h-quasiconvex function $\ul{f}\in C(\Oba)$ such that 
$\ul{f}\leq f$ in $\Oba$ and $\ul{f}=K$ on $\partial \Omega$; see Proposition \ref{prop lower}.
This in turn implies the existence of $Q(f)$ taking the same boundary value. 

We set $u_0=f$ in $\Omega$ and take $u_n$ ($n=1, 2, \ldots$) to be the unique viscosity solution of 
\beq\label{iteration eq}
u_n+H(p, u_n, \nabla_H u_n)=u_{n-1}  \quad \text{in $\Omega$}
\eeq
satisfying the same set of conditions as in \eqref{dirichlet-f}--\eqref{coercive-f}, that is, 
\begin{numcases}{}
u_n=K \quad &\text{on $\pO$,}\label{data bdry} \\
u_n\leq K \quad &\text{in $\Oba$.}\label{data bdry ineq}
\end{numcases}

It turns out that $u_n$ is a nonincreasing sequence and converges uniformly to $Q(f)$ as $n\to \infty$. This is in fact our main theorem. 

\begin{thm}[Iterative scheme for envelope]\label{thm scheme bdry}
Suppose that $\Omega$ is a bounded h-convex domain in $\H$ and $f\in C(\Oba)$ satisfies \eqref{dirichlet-f}--\eqref{coercive-f} for some $K\in \R$. 
Let $Q(f)\in USC(\Omega)$ be the h-quasiconvex envelope of $f$ in $\Omega$. Let $u_0=f$ in $\Omega$ and $u_n$ be the unique solution of \eqref{iteration eq} satisfying \eqref{data bdry}\eqref{data bdry ineq} for $n\geq 1$. Then $u_n\to Q(f)$ uniformly in $\Oba$ as $n\to \infty$. %In addition, if $Q(f)\in C(\Oba)$, then $u_n\to Q(f)$ uniformly in $\Oba$ as $n\to \infty$. 
\end{thm}
Such type of nonlocal schemes is proposed in \cite{BGJ1} in the Euclidean case for general Dirichlet data. We remark that in the Euclidean space a similar class of nonlocal equations depending on the level sets of the unknown is also studied for applications in geometric evolutions and front propagation \cite{Car1, Car2, Sl, KLM}. 
Although our PDE looks analogous to theirs, the well-posedness in the sub-Riemannian case is not straightforward at all. The main difference from the Euclidean case lies in an additional constraint that requires $\xi\in \H_p\cap \Omega$, which depends on the space variable $p$. This extra constraint brings us much difficulty in proving the comparison principle for \eqref{iteration eq}. It is the coercivity-like setting \eqref{data bdry}\eqref{data bdry ineq} that enables us to overcome the difficulty and obtain the uniqueness of solutions. More details can be found in Section \ref{sec:comp}. 

The existence of viscosity solutions, on the other hand, can be handled in a standard way by adapting Perron's method \cite{CIL}. Since the existence in the Euclidean case is not explicitly discussed in \cite{BGJ1}, we give full details of the arguments for our sub-Riemannian version in Section \ref{sec:exist}. Once the sequence $u_n$ is determined iteratively for all $n=1, 2, \ldots$, Theorem \ref{thm scheme bdry} can be proved by applying a stability argument for viscosity solutions.

It is worth mentioning that, instead of adopting the restrictive setting \eqref{data bdry}\eqref{data bdry ineq}, one can solve \eqref{iteration eq} with general boundary data  and obtain the scheme convergence as in Theorem \ref{thm scheme bdry} if there exists an appropriate h-quasiconvex barrier $\ul{f}$ from below compatible with the boundary value; see Theorem \ref{thm existence2} and Remark \ref{rmk convergence2}. 

Another possible modification of the scheme is to consider the maximal  subsolution of \eqref{iteration eq} rather than its solutions at each step. Although we only get pointwise convergence of the scheme in this case, it allows us to avoid the uniqueness issue and to construct the h-quasiconvex envelopes for a general class of upper semicontinuous functions, even in an unbounded domain $\Omega$. See Theorem \ref{thm scheme} for results in this relaxed setting. 

We would like to point out that, besides the PDE-based approach described above, there is a more direct constructive method to build $Q(f)$, which employs the following convexification operator:
 \beq\label{direct conv}
T[f](w) = \inf  \left\{ \max \{ f(p), f(q)\}: w \in [p, q], \ p \in \Omega,\ q \in \Omega \cap \H_p \right\}, \   \text{for }w\in \Omega.
\eeq
It turns out that $T[f]$ itself may not be h-quasiconvex in $\Omega$ but iterated application of $T$ yields a pointwise approximation of the h-quasiconvex envelope $Q(f)$; see Theorem \ref{thm direct} for details. A similar idea is used in the proof of \cite[Lemma 4.1]{Riconv} for h-convex functions. It is also adopted in \cite{LZ2} to construct the h-convex envelope of a given function $f$. 

As an application of our constructive methods for h-quasiconvex envelopes, we study the h-convex hull of a given bounded set in $\H$. A key ingredient is the so-called level set formulation, which plays an important role in the study of geometric evolutions \cite{CGG, ES1, Gbook}. We can apply the same idea to our problem, since the nonlocal Hamiltonian is actually a geometric operator, homogeneous in $u$. 
Suppose that $E$ is a bounded open set in $\H$. Take a bounded h-convex domain $\Omega$ such that $E\subset \Omega$. We next choose a defining function $f\in C(\Oba)$ such that 
\beq\label{initial defining func}
E=\{p\in \Omega: f(p)<0\}
\eeq
and \eqref{dirichlet-f}\eqref{coercive-f} hold for some $K>0$.
It turns out that the h-convex hull of $E$ coincides with the zero sublevel set of $Q(f)$, i.e., 
\beq\label{final defining func}
\coh(E)=\{p\in \Omega: Q(f)<0\}.
\eeq
We remark that $\coh(E)$ is independent of the choices of $f$ and $\Omega$. As long as \eqref{initial defining func} together with \eqref{dirichlet-f}\eqref{coercive-f} holds, $\coh(E)$ obtained in \eqref{final defining func} will not change. See Theorem \ref{thm hull1} for more precise statements. 

This PDE approach leads us to a better understanding about h-convex hulls. One application is about the inclusion principle. By definition, it is easily seen that $\coh(D)\subset \coh(E)$ holds for any sets $D, E\subset \H$ satisfying $D\subset E$. In Theorem \ref{thm sep}, we establish a quantitative version of the inclusion principle in $\H$. For any bounded open (or closed) sets $D, E\subset \H$, we obtain 
\beq\label{dist sep}
\inf\left\{\tilde{d}_H(p, q): p\in \coh(D), q\in \H\setminus \coh(E) \right\}\geq \inf\left\{\tilde{d}_H(p, q): p\in D, q\in \H\setminus E\right\}.
\eeq
where $\tilde{d}_H$ denotes the right invariant gauge metric in $\H$; see \eqref{right metric} below. This property amounts to saying that taking h-convex hulls of two sets in $\H$, one contained in the other, does not reduce the shortest $\tilde{d}_H$ distance between their boundaries. If $E$ contains the right invariant $\delta$-neighborhood of $D$ for some $\delta>0$, then $\coh(E)$ also contains the right invariant $\delta$-neighborhood of $\coh(D)$.

While such a result can be obtained comparatively easily in the Euclidean case, the proof is more involved in the Heisenberg group. 
Our proof is based on comparing the h-quasiconvex envelopes of defining functions of both sets combined with arguments involving sup-convolutions. It is not clear to us whether one can replace $\tilde{d}_H$ by the left invariant gauge metric $d_H$. This problem is related to the h-convexity preserving property for solutions of evolution equations in the Heisenberg group; see some partial results in \cite{LMZ, LZ2}. 

Another natural question is on the continuity (or stability) of $\coh(E)$ with respect to the set $E$, which we discuss in the last part of this paper. In contrast to the Euclidean case, in general the Hausdorff distance $d_H(\coh(E_j), \coh(E))$ between $\coh(E_j)$ and $\coh(E)$ does not necessarily converge to zero when $d_H (E_j,E)\to 0$ in the Heisenberg group, see Example \ref{ex2}. One can show rather easily that such stability result holds under a strict star-shapedness assumption on the set $E$; see Proposition \ref{prop stable}.

\subsection{Notations}
We conclude the introduction by listing several notations that are often used in the work. Throughout this paper,  $|\cdot |_G$ stands for the Kor\'{a}nyi gauge, i.e., for $p=(x, y, z)\in \H$
\[
|p|_G=\left((x^2+y^2)^2+16z^2\right)^{\frac{1}{4}}.
\]
The Kor\'{a}nyi gauge induces a left invariant metric $d_H$ on $\H$ with
\[
d_H(p, q)=|p^{-1}\cdot q|_G\quad p, q\in \H.
\]
We also use the right invariant metric $\tilde{d}_H$,  defined by 
\beq\label{right metric}
\tilde{d}_H(p, q)=|p\cdot q^{-1}|_G, \quad p, q\in \H.
\eeq
The associated distances between a point $p\in \H$ and a set $E\subset \H$ are respectively denoted by $d_H(p, E)$ and $\tilde{d}_H(p, E)$. 
For two sets $D, E\subset \H$, we write $d_H(D, E)$ and $\tilde{d}_H(D, E)$ to denote respectively the Hausdorff distances between $D$ and $E$ with respect to the metrics $d_H$ and $\tilde{d}_H$, i.e, for $d=d_H$ or $d=\tilde{d}_H$,
\[
d(D, E)=\max\left\{\sup_{p\in D}d(p, E),\ \sup_{p\in E} d(p, D)\right\}.
\]

We denote by $B_r(p)$ the open gauge ball in $\H$ centered at $p\in \H$ with radius $r>0$, that is, 
\[
B_r(p)=\{q\in \H: |p^{-1}\cdot q|_G< r\},
\]
while $\tilde{B}_r(p)$ represents the corresponding right-invariant metric ball. 

Let $\delta_\lambda$ denote the non-isotropic dilation in $\H$ with $\lambda\geq 0$, that is, $\delta_\lambda(p)=(\lambda x, \lambda y, \lambda^2 z)$ for $p=(x, y, z)\in \H$. 
We write $\delta_\lambda(E)$ to denote the dilation of a given set $E\subset \H$, that is, $\delta_\lambda(E)=\{\delta_\lambda(p): p\in E\}$.

The rest of the paper is organized in the following way. In Section \ref{sec:h-quasiconvex}, we first give a review on the definitions and basic properties of h-convex sets and h-quasiconvex functions, and then present the viscosity characterization of upper semicontinuous h-quasiconvex functions. We also show how to construct the h-quasiconvex envelope by iterated application of the operator in \eqref{direct conv}.
Section \ref{sec:nonlocal-hj} is devoted to the well-posedness of the nonlocal Hamilton-Jacobi equation, including the uniqueness and existence of viscosity solutions. Our PDE-based iterative scheme is introduced in Section \ref{sec:iteration}. We finally discuss applications of our results to the h-convex hull of a given open or closed set in Section \ref{sec:convex-hull}.

\section{H-quasiconvex functions}\label{sec:h-quasiconvex}

\subsection{Definition and basic properties}

Let us first go over the definition of h-convex sets. We restrict the original definition proposed in \cite{DGN1} for general Carnot groups to the case of $\H$. 

\begin{defi}[Definition 7.1 in \cite{DGN1}]\label{def h-set}
We say, that a set $E \subset \H$ is h-convex if for every $p\in E$ and $q\in \H_p\cap E$, the horizontal segment $[p, q]$ joining $p$ and $q$ stays in $E$.
\end{defi}

As pointed out in \cite[Proposition 7.4]{DGN1}, any gauge ball $B_R(p)$ with $p\in \H$ and $R>0$ is h-convex. The notion of h-convex sets is in fact very weak. There are numerous h-convex sets in $\H$ that are obviously not convex in the Euclidean sense. 

\begin{example}[Disconnected h-convex sets]\label{ex1}
Denote by $\pi(0, \rho)$ the planar open disk centered at the origin with radius $\rho>0$, i.e., 
\begin{equation}\label{h-pi}
\pi(0, \rho) := \{(x,y) \in \R^2: x^2+y^2 < \rho^2 \}.
\end{equation}
Let us consider a disconnected set $E =(\pi(0, r) \times \{0\}) \cup (\pi(0,R) \times \{t\})$, where $r, R, t>0$ are given. Such a set $E$ is h-convex under appropriate conditions on $r$ and $R$.  To see this, we take the horizontal plane 
 \[
Z_t=\{(x, y, z)\in \H: z=t\}
 \]
 and compute the distance between $q_t=(0, 0, t)$ and $\H_p\cap Z_t$ for each point $p=(x_p, y_p, 0)\in \pi(0, r)\times \{0\}$. It turns out that 
 \[
 d_H(q_t, \H_p\cap Z_t) = \frac{2 t}{\sqrt{x_p^2 + y_p^2}} \geq \frac{2t}{r}.
 \]
If $d_H(q_t, \H_p\cap Z_t)\geq R$, then none of the horizontal planes of points $p\in \pi(0, r)\times \{0\}$ in the lower disk will intersect the upper disk $\pi(0, R)\times \{t\}$. This means that $E$ is h-convex if $2t\geq rR$. It is obvious that in general $E$ is not connected and thus cannot be convex as a subset of $\R^3$. It is also clear that $E$ is no longer h-convex if $2t<rR$.
\end{example}

Let us also recall from \cite{CCP2} the definition of h-quasiconvex functions in $\H$. 

\begin{defi}[Definition 4.3 in \cite{CCP2}]\label{def h-fun}
Suppose that $\Omega \subset \H$ is h-convex. We say, that a function $u :\Omega \to \R$ is h-quasiconvex if \eqref{h-quasi def} holds
for every $p\in \Omega$, $q \in \H_p\cap \Omega$ and $w \in [p,q]$. In other words, $u$ is h-quasiconvex if for every $\lambda \in \R$ the sublevel set $\{ w \in \Omega: u(w) \leq \lambda\}$ is h-convex.
\end{defi}

\begin{rmk}
We remark that it is equivalent to define h-convex functions with strict sublevel set $\{w\in \Omega: u(w)<\lambda\}$. In fact, first note that the 
\[
\{w\in \Omega: u(w)\le \lambda\}=\bigcap_{\vep>0}\{w\in \Omega: u(w)< \lambda+\vep\}
\] 
and the intersection of h-convex sets are still h-convex. On the other hand, if $\{w\in \Omega: u(w)\le \lambda\}$ is h-convex, then \eqref{h-quasi def} holds for every $p\in \Omega$, $q \in \H_p\cap \Omega$ and $w \in [p, q]$. If $p, q\in \{w\in \Omega: u(w)<\lambda\}$ and $q \in \H_p\cap \Omega$, it follows that $u(w)<\lambda$ when $w \in [p, q]$ and thus $\{x\in \Omega: u(w)<\lambda\}$ is h-convex. 
\end{rmk}

For our later applications, below we provide a typical h-quasiconvex function associated to a given h-convex set. The construction is based on the right invariant metric $\tilde{d}_H$, as given in \eqref{right metric}.

\begin{prop}[A metric-based h-quasiconvex function]\label{prop metric-quasi}
Suppose that $\Omega\subset \H$ is an h-convex domain in $\H$ and $E$ is an h-convex open subset of $\Omega$. Then $\psi_{E}\in C(\Omega)$ given by
\[
\psi_E(p)=-\tilde{d}_H(p, \H\setminus E), \quad p\in \Omega, 
\]
is an h-quasiconvex function in $\Omega$. 
\end{prop}
\begin{proof}
It suffices to show that 
\[
E_\lambda:=\{p\in \Omega: \psi_E(p)<\lambda\}
\] 
is h-convex for every $\lambda\in \R$. We only need to consider the case for $\lambda< 0$, since $E_\lambda=E$ if $\lambda=0$ and $E_\lambda=\Omega$ if $\lambda> 0$. Assume by contradiction that there exist $p, q\in E_\lambda$ with $q\in \H_p$ as well as a point $w\in [p, q]\cap (\Omega\setminus E_\lambda)$. This means that 
\[
\tilde{d}_H(w, \H\setminus E)\le-\lambda
\]
and there exists $\zeta\in \H\setminus E$ such that $|\zeta\cdot w^{-1}|_G\leq -\lambda$. 
Taking $\xi=\zeta\cdot w^{-1}\cdot p, \quad \eta=\zeta\cdot w^{-1} \cdot q$, we can easily verify that $\eta\in \H_\xi$ and $\zeta\in [\xi, \eta]$. 

Suppose that $\xi \in \H\setminus E$ holds. Then, since 
\[
\tilde{d}_H(\xi, p)=|\xi\cdot p^{-1}|_G=|\zeta\cdot w^{-1}|_G\leq -\lambda,
\]
we have $\tilde{d}_H(p, \H\setminus E)\leq -\lambda$ or equivalently $\psi_E(p)\geq \lambda$, which is a contradiction to the condition $p\in E_\lambda$. We can similarly derive a contradiction if $\eta\in \H\setminus E$ holds. 

The remaining case when $\xi, \eta\in E$ is an obvious contradiction to the assumption that $E$ is h-convex, since $\zeta\in [\xi, \eta]$ and $\zeta\in \H\setminus E$.
\end{proof}

\subsection{Viscosity characterization of h-quasiconvexity}

The following characterization of h-quasiconvexity is known for smooth functions \cite[Theorem 4.5]{CCP2}.  We provide a generalized result in the nonsmooth case by extending \cite[Proposition 2.2]{BGJ1} to the Heisenberg group. 

\begin{thm}[Viscosity characterization of h-quasiconvexity]\label{thm char}
Let $\Omega \subset \H$ be open and h-convex and $u :\Omega \to [-\infty,\infty)$ be upper semicontinuous. Then, $u$ is h-quasiconvex if and only if whenever there exist $p\in \Omega$ and $\varphi\in C^1(\Omega)$ such that $u-\varphi$ attains a maximum  at $p$, 
\begin{equation}\label{usc}
\langle \nabla_H \varphi(p), (p^{-1}\cdot \xi)_h \rangle \leq 0\quad \text{holds for any $\xi \in \H_p\cap \Omega$ satisfying $u(\xi) < u(p)$}.
\end{equation}
\end{thm}

\begin{rmk}
It is worth mentioning that, as in \cite{CCP2},  we can also express the inner product term $\la \nabla_H \varphi(p), (p^{-1}\cdot \xi)_h\ra$ in \eqref{usc} by $\la \nabla \varphi(p), \xi-p\ra$. 
This is possible because of the condition that $\xi\in \H_p$. We shall maintain the expression on the left hand side to suggest possible generalization of our results in general Carnot groups, which is not elaborated in this paper. 
\end{rmk}

\begin{proof}[Proof of Theorem \ref{thm char}]
Let us prove the necessity of \eqref{usc} by contradiction. Suppose, that $u$ is upper semicontinuous h-quasiconvex function, $\varphi\in C^1(\Omega)$ is such that $u-\varphi$ attains a maximum at $p\in \Omega$, and there exists $\xi\in \H_p\cap \Omega$ with $u(\xi) < u(p)$ such that 
\[
\langle \nabla_H \varphi(p) , (p^{-1}\cdot \xi)_h \rangle>0. 
\]
Then, for $\lambda>0$ small enough and $w = p\cdot \delta_\lambda (\xi^{-1}\cdot p)$ there holds $u(w) < u(p)$. Indeed, the directional derivative of $\varphi$ at $p$ in the direction $p^{-1}\cdot \xi$ is positive, and hence $\varphi(w) < \varphi(p)$ for $\lambda>0$ small enough. Since $u-\varphi$ attains a maximum at $p$ we obtain $u(p)>u(w)$. We conclude with $\xi \in \H_w$, $p \in [w, \xi]$ and $u(p)>\max\{u(w),u(\xi)\}$, which contradicts the h-quasiconvexity of $u$.

Now we are left with proving sufficiency of \eqref{usc}. Suppose that $u$ is not h-quasiconvex. Then, without loss of generality there exists a point $\xi=(x_\xi, y_\xi, 0) \in \H_0$  such that 
\[
\max_{[0,\xi]} u > \max\{u(0),u(\xi)\}.
\] 
Denote by $Z$ the set of maximizers of $u$ on the segment $[0,\xi]$. By the upper semicontinuity of $u$ there exists $R \in (0, |\xi|/4)$ small enough such that 
\[
\min \{d_H(0,Z),d_H(\xi,Z)\}>R
\] 
(i.e. $Z$ is in the relative interior of $[0,\xi]$) and $ u(q) < \max_{[0, \xi]} u$ for any $q \in B_R(0) \cup B_R (\xi)$. Let 
\[
\mathcal{C}:=\{q \in \Omega: d_H(q,[0,\xi])<R, 0<\langle q,\xi \rangle < |\xi|^2 \}
\] 
be a cylindrical neighborhood of the segment $[0,\xi]$. We define $\varphi_n$ by
\[
\varphi_n(p) = \frac{1}{n} \langle p, \xi \rangle +n \left((x_py_\xi-y_px_\xi)^2+z_p^2\right), \quad p\in \Omega. 
\]
As $n\to \infty$, $u-\varphi_n\to u$ pointwise in $[0,\xi]$ and $u-\varphi_n\to -\infty$ elsewhere in $\overline{\mathcal{C}}$. 
Then there exists a sequence $p_n=(x_{p_n},y_{p_n},z_{p_n}) \in \overline{\mathcal{C}}$ such that
\[ \max_{\overline{\mathcal{C}}} (u-\varphi_n ) =u(p_n)- \varphi_n(p_n),
\]
 which converges to a point in $Z$ via a subsequence. Let us index the subsequence still by $n$ for notational simplicity. Suppose that the subsequence $p_n \to (tx_\xi, ty_\xi, 0)$ as $n \to \infty$ for some $t \in (R/|\xi|,1-R/|\xi|)$. Let us consider $w_n =(x_{p_n}/ t, y_{p_n}/t ,z_{p_n}) \in \H_{p_n}$. Observing that 
\[
d_H(\xi,w_n) = |\xi^{-1} \cdot w_n |_G \to 0,
\] 
we get $u(w_n)<u(p_n)$ for $n$ large enough.

Let us compute $\nabla_H \varphi_n(p_n)$ as follows:
\[
\nabla_H \varphi_n (p_n) = \Bigg( \frac{x_\xi }{n} +2n y_\xi (x_{p_n}y_\xi -y_{p_n}x_\xi)- n y_{p_n} z_{p_n}, \frac{y_\xi }{n} -2nx_\xi (x_{p_n}y_\xi -y_{p_n}x_\xi)+ n x_{p_n} z_{p_n}\Bigg).
\]
Since 
\[
(p_n^{-1}\cdot w_n)_h = \left( \frac{1}{t}-1 \right) (x_{p_n}, y_{p_n}),
\]
we have
\[\langle \nabla_H \varphi_n(p_n),(p_n^{-1}\cdot w_n)_h \rangle = \left(\frac{1}{t}-1\right) \left( \frac{\langle p_n,\xi \rangle}{n} +2n (x_{p_n}y_\xi -y_{p_n}x_\xi)^2 \right)>0,
\]
which contradicts \eqref{usc}.
\end{proof}

Theorem \ref{thm char} amounts to saying that $u\in USC(\Omega)$ is h-quasiconvex if \eqref{vis char}
holds in the viscosity sense, that is, 
\[
\sup\{\la \nabla_H \varphi(p), (p^{-1}\cdot \xi)_h\ra: \xi\in \H_p\cap \Omega, \ u(\xi)< u(p)\}\leq 0
\]
whenever there exist $p\in \Omega$ and $\varphi\in C^1(\Omega)$ such that $u-\varphi$ attains a maximum at $p$. As a standard remark in the viscosity solution theory, the maximum here can be replaced by a local maximum or a strict maximum.

In spite of the nonlocal nature, we can obtain the following property by using the geometricity of the operator. 
\begin{lem}[Invariance with respect to composition]\label{lem geometric}
Let $\Omega\subset \H$ be h-convex and $u: \Omega\to \R$ be bounded and upper semicontinuous. Assume that $g: \R\to \R$ is a nondecreasing continuous function. If $u$ is h-quasiconvex in $\Omega$, then so is $g\circ u$. 
\end{lem}
\begin{proof}
It is clear that $g\circ u$ is still upper semicontinuous. 
Assume first that $g\in C^1(\R)$ is strictly increasing. It is clear that $g\circ u$ is still upper semicontinuous. Suppose that there exists $p\in \Omega$ and $\varphi\in C^1(\Omega)$ such that $g\circ u-\varphi$ attains a maximum at $p_0$. Then using the inverse function $g^{-1}$ of $g$, we see that $u-g^{-1}\circ\varphi$ also attains a maximum at $p_0$. Applying the characterization of h-quasiconvexity in Theorem \ref{thm char}, we get
\[
\sup\{\la \nabla_H (g^{-1}\circ \varphi)(p_0), (p_0^{-1}\cdot \xi)_h\ra: \xi\in \H_{p_0}\cap \Omega, \ u(\xi)< u(p_0)\}\leq 0,
\]
which implies 
\beq\label{geometric eq1}
\sup\{\la \nabla_H \varphi(p_0), (p_0^{-1}\cdot \xi)_h\ra: \xi\in \H_{p_0}\cap \Omega, \ (g\circ u)(\xi)< (g\circ u)(p_0)\}\leq 0.
\eeq
This shows that $g\circ u$ is h-quasiconvex. 

For the general case when $g\in C(\R)$ is nondecreasing, we can take a sequence $g_k\in C^1(\R)$ increasing such that $g_k\to g$ locally uniformly in $\R$ as $k\to \infty$. In this case, suppose that there exists $p\in \Omega$ and $\varphi\in C^1(\Omega)$ such that $g\circ u-\varphi$ attains a strict maximum at $p_0$. Then there exist $p_k\in \Omega$ such that $p_k\to p_0$ as $k\to \infty$ and $g_k\circ u-\varphi$ attains a local maximum at $p_k$. 

Adopting the argument above, we obtain the h-quasiconvexity of $g_k\circ u$, that is, 
\[
\sup\{\la \nabla_H \varphi(p_k), (p_k^{-1}\cdot \xi)_h\ra: \xi\in \H_{p_k}\cap \Omega, \ (g_k\circ u)(\xi)< (g_k\circ u)(p_k)\}\leq 0.
\]
By the uniform convergence of $g_k\circ u$ to $g\circ u$ in $\Omega$ and the continuity of $\H_p$ with respect to $p$, we can pass to the limit as $k\to \infty$ and obtain the relation \eqref{geometric eq1} again. 
\end{proof}

\subsection{H-quasiconvex envelope}

In what follows, we introduce the h-quasiconvex envelope of a given function. 

\begin{defi}[Definition of h-quasiconvex envelope]\label{defi envelope}
Let $\Omega$ be an h-convex domain in $\H$ and $f: \Omega \to \R$ be a given function. We say, that $Q(f)$ is the h-quasiconvex envelope of $f$ if it is the greatest h-quasiconvex function majorized by $f$, that is 
\[ 
Q(f)(p) := \sup \{g(p): g \leq f  \text{ and } g \text{ is h-quasiconvex}\},
\]
where we adopt the convention that $\sup \emptyset = -\infty$.
\end{defi}

By definition, $Q(f)$ is monotone in $f$; namely, $Q(f)\leq Q(g)$ in $\Omega$ holds provided that $f\leq g$ in $\Omega$.
Moreover, $Q(f)$ is stable with respect to $f$ in the following sense.
\begin{prop}[Stability of h-quasiconvex envelope]\label{prop stability}
Suppose that $\Omega$ is an h-convex domain in $\H$ and $f, g: \Omega\to \R$ are given functions. Assume that both $Q(f)$ and $Q(g)$ exist in $\Omega$. Then there holds 
\beq\label{fun stable}
\sup_{\Omega} |Q(f)-Q(g)|\leq \sup_{\Omega}|f-g|.
\eeq
\end{prop}
\begin{proof}
Let $M:=\sup_{\Oba}|f-g|$. Since $f-M\leq g$ in $\Oba$,
by the monotonicity of $Q$, we get 
\beq\label{fun stable1}
Q(f-M)\leq Q(g) \quad \text{in $\Omega$.}
\eeq
Noticing that $Q(f)-M$ is h-quasiconvex and $Q(f)\leq f$ in $\Omega$, 
by Definition \eqref{defi envelope}, we deduce that
\[
Q(f)-M\leq Q(f-M)\quad \text{in $\Omega$},
\]
which, by \eqref{fun stable1}, yields
\[
Q(f)-M\leq Q(g)\quad \text{in $\Omega$}.
\]
Exchanging the roles of $f$ and $g$, we conclude the proof of \eqref{fun stable}.
\end{proof}

Let us now discuss how to find the h-convex envelope of a given function. A straightforward method is to employ a convexification operator.  For an h-convex domain $\Omega \subset \H$ and $f: \Omega \to \R$, let $T[f]$ be given by \eqref{direct conv},
It is clear that $\inf_{\Omega} f \leq T[f] \leq f$ in $\Omega$. Also, it is easily seen that $T[f]=f$ in $\Omega$ if and only if $f$ is h-quasiconvex. 

This operator is inspired by its Euclidean analogue, which is given by 
\[
T_{eucl}(w)= \inf  \left\{ \max \{ f(p),f(q)\}: w \in [p, q],  p \in \Omega, q \in \Omega \right\},\quad w\in \R^3.
\]
In the Euclidean case, where the quasiconvex envelope, written as $Q_{eucl}(f)$,  satisfies
\[
Q_{eucl}(f)=T_{eucl}[f]
\]
in a bounded convex domain $\Omega$. In contrast, the following example shows that in the Heisenberg group, in general $T[f]$ is not necessarily an h-quasiconvex function.

\begin{example}\label{exa direct}
Let $f:\H\to \R$ be defined as $f(p)=|1-z^2|$ for $p=(x, y, z)\in \H$. One can compute directly to get, for $p=(x, y, z)$,  
\[
T[f](p)=\begin{cases} 
z^2-1&\  |z|\ge 1, \\
0&\  |z|<1\ \text{and}\ (x,y)\neq(0,0),\\
1-z^2 &\  |z|<1\ \text{and}\ (x, y)=(0,0),\\
\end{cases}
\]
which fails to be quasiconvex. In fact, letting $p=(x, y, t)$, $q=(-x, -y, t)$ with $|t|<1$, we see that  $q\in \mathbb{H}_p$ and at $w=(0,0, t)\in [p, q]$, we have 
\[
T[f](w)=1-t^2>0=\max\{T[f](p), T[f](q)\}. 
\]
However, if we apply the operator one more time, then we have, for $p=(x, y, z)\in \H$, 
\[
T^2[f](p)=\begin{cases} z^2-1&\ |z|\ge 1, \\
0&\ |z|<1.
\end{cases}
\]
It is not difficult to see that $Q(f)=T^2[f]$ in $\H$. Indeed, Noticing that $T^2[f]$ is h-quasiconvex, by definition we have $T^2[f]\leq Q(f)$ in $\H$. On the other hand, the reverse inequality $Q(f)\leq T^2[f]$ can be obtained by applying the operator $T$ twice to the inequality $Q(f)\leq f$. 
\end{example}

It turns out that in general one can obtain the quasiconvex envelope by iterating the operator $T$. Such type of iteration is also used in \cite{LZ2} to construct the h-convex envelope of a given continuous function in the Heisenberg group. 
\begin{thm}[Iterative scheme with direct convexification]\label{thm direct}
Let $\Omega$ be an h-convex domain in $\H$. Suppose that $f$ is bounded from below. Let $T$ be the operator given by \eqref{direct conv}. Then $T^n[f] \to  Q(f)$ pointwise in $\Omega$ as $n\to \infty$.
\end{thm}

\begin{proof}
Notice that by the monotonicity of $T^n[f]$ in $n$ and boundedness of $f$ from below,  the pointwise limit of $T^n[f]$ exists. Let us denote it by $F$, i.e., 
\[
F:=\lim_{n\to \infty} T^n[f]. 
\] 
Let us fix $\varepsilon>0$, $p \in \Omega$, $q \in \Omega \cap \H_p$ and $w \in [p,q]$. For $n$ sufficiently large, there holds 
\[ 
F(p) \geq T^n[f](p) - \varepsilon, \qquad F(q) \geq T^n[f](q) - \varepsilon.
\]
Moreover, we have
 \[
  \max \{ T^n[f](p),T^n[f](q) \} \geq T^{n+1} [f](w) \geq F(w),
 \]
and therefore
 \[ 
 \max\{ F(p), F(q) \} \geq F(w) -\varepsilon.
 \]
Letting $\varepsilon \to 0$, we deduce that $F$ is h-quasiconvex and thus $F \leq Q(f)$ in $\Omega$. 

On the other hand, for any $w \in \Omega$ and $p \in \Omega$, $q \in \Omega \cap \H_p$ such that $w\in [p, q]$,  there holds
\[ 
\max \{ f(p), f(q)\} \geq \max\{ Q(f)(p),Q(f)(q)\} \geq Q(f)(w).
\]
It follows that $T[f] \geq Q(f)$ in $\Omega$. 
We can iterate this argument obtain 
$T^n[f] \geq Q(f)$ in $\Omega$ for every $n$. Hence, sending $n\to \infty$, we are led to $F\geq Q(f)$ holds in $\Omega$, which completes the proof.
\end{proof}

As shown in Example \ref{exa direct}, $T^n[f]=Q(f)$ may hold for a finite $n$. We do not know in general how many iterations one needs to run to obtain $Q(f)$. It would be interesting to find a condition to guarantee the finiteness of $n$.

\section{Nonlocal Hamilton-Jacobi equation}\label{sec:nonlocal-hj}
Inspired by \cite{BGJ1}, we would like to focus our attention on a PDE-based approach to build the h-quasiconvex envelope of a given function $f$ in a bounded domain $\Omega\subset \H$. We develop the idea in Theorem \ref{thm char}, which provides a characterization of h-quasiconvexity in terms of viscosity subsolutions of a nonlocal PDE. 
For our convenience of notations below, for any function $u: \Omega\to \R$ and any $p\in \Omega$, we denote 
\[
S_p(u)=\{\xi\in \H_p\cap \Omega: u(\xi)<u(p)\}.
\]

We study the following nonlocal Hamilton-Jacobi equation:
\begin{equation}\label{hj eq}
u(p)+H(p, u(p), \nabla_H u(p))=f(p) \quad \text{in $\Omega$},
\end{equation} 
to which the subsolutions are defined with
\[
H(p, u(p), \nabla_H u(p))=\sup\left\{\la \nabla_H u(p),  (p^{-1}\cdot \xi)_h\ra:  \xi\in S_p(u)\right\}
\]
while the supersolutions are defined with
\[
H(p, u(p), \nabla_H u(p))=\sup\left\{\la \nabla_H u(p),  (p^{-1}\cdot \xi)_h\ra:  \xi\in \hat{S}_p(u)\right\}.
\]
Here we set
\[
\hat{S}_p(u):=\{\xi\in \H_p\cap \Omega: u(\xi)\leq u(p)\}.
\]

The major difficulty lies at the degeneracy and nonlocal nature of the first order operator.
We mainly study uniqueness and existence of solutions to this equation in a slightly restrictive setting, assuming that the solutions $u$ satisfy
\begin{numcases}{}
u=K \quad &\text{on $\pO$, }\label{dirichlet} \\
u\leq K \quad &\text{in $\Oba$}\label{coercive-like}
\end{numcases}
for some $K\in \R$. It turns out that we can obtain a unique solution if $f$ satisfies the same conditions. 

These conditions can be viewed as a bounded-domain variant of coercivity assumption on $u$. If $u\in C(\H)$ is coercive in $\H$, that is, 
\[
\lim_{R\to \infty} \inf_{|p|_G\geq R} u(p)\to \infty,
\]
then we can take $K\in \R$ large and $\Omega=\{p: u(p)<K\}$ such that both \eqref{dirichlet} and \eqref{coercive-like} hold. 

\subsection{Definition and basic properties of solutions}

Let us first present the definition of subsolutions of \eqref{hj eq}.

\begin{defi}[Subsolutions]\label{def sub}
Let $f\in USC(\Omega)$ be locally bounded in $\Omega$. A locally bounded upper semicontinuous function $u: \Omega\to \R$ is called a subsolution of \eqref{hj eq} if whenever there exist $p\in \Omega$ and $\varphi\in C^1(\Omega)$ such that $u-\varphi$ attains a maximum at $p$, 
\beq\label{sub-property}
u(p)+\sup\left\{\la \nabla_H \varphi(p),  (p^{-1}\cdot \xi)_h\ra:\ \xi\in S_p(u) \right\}\leq f(p).
\eeq
\end{defi}
Note that the supremum in \eqref{sub-property} does make sense. We consider naturally
\[
\sup\left\{\la \nabla_H \varphi(p),  (p^{-1}\cdot \xi)_h\ra:\ \xi\in S_p(u) \right\}=0
\]
if $\nabla_H \varphi(p)=0$. It is also easily seen that $S_p(u)\neq \emptyset$ provided that $\nabla_H\varphi(p)\neq 0$ and $u-\varphi$ attains a local maximum at $p$. 

We next give a definition of supersolutions. 

\begin{defi}[Supersolutions]\label{def super1}
Let $f\in LSC(\Omega)$ be locally bounded in $\Omega$. 
A locally bounded lower semicontinuous function $u: \Omega\to \R$ is called a weak supersolution of \eqref{hj eq} if whenever there exist $p\in \Omega$ and $\varphi\in C^1(\Omega)$ such that $u-\varphi$ attains a minimum at $p$, 
\beq\label{super-property}
u(p)+\sup\left\{\la \nabla_H \varphi(p),  (p^{-1}\cdot \xi)_h\ra:\ \xi\in \hat{S}_p(u) \right\}\geq f(p).
\eeq
\end{defi}

For any point $p\in \Omega$, we also say that $u\in USC(\Omega)$ (resp, $LSC(\Omega)$) satisfies the subsolution (resp., supersolution) property at $p$ if \eqref{sub-property} (resp., \eqref{super-property}) holds for any $\varphi\in C^1(\Omega)$ such that $u-\varphi$ attains a maximum (resp., minimum) at $p$.

As a standard remark in the theory of viscosity solutions, we may replace the maximum in the definitions above by local maximum or strict maximum. 

\begin{lem}[Upper bound]\label{lem sub}
Suppose that $\Omega$ is a domain in $\H$. Let $f\in USC(\Omega)$ be locally bounded in $\Omega$. If $u\in USC(\Omega)$ is a subsolution of \eqref{hj eq}. Then $u\leq f$ in $\Omega$.
\end{lem}

\begin{proof}
Let us first show $u\leq f$ at all points where $u$ can be tested. Assume that there exists $\varphi\in C^1(\Omega)$ such that $u-\varphi$ attains a maximum at some $p_0\in \Omega$. If $\nabla_H\varphi(p_0)=0$, then we immediately obtain the desired inequality $u(p_0)\leq f(p_0)$ by Definition \ref{def sub}. 

If $\nabla_H\varphi(p_0)\neq 0$ and thus $S_{p_0}(u)\neq \emptyset$, we can take a sequence $\xi_j\in S_{p_0}(u)$ such that $\xi_j\to p_0$ as $j\to \infty$. This yields
\[
\sup\left\{\la \nabla_H \varphi(p_0),  (p_0^{-1}\cdot \xi)_h\ra:\ \xi\in S_{p_0}(u) \right\}
\geq \limsup_{j\to \infty}\la \nabla_H \varphi(p_0),\ (p_0^{-1}\cdot \xi_j)_h\ra=0,
\]
As a result, we get $u(p_0)\leq f(p_0)$ by Definition \ref{def sub} again. 

It remains to show that $u\leq f$ holds also at those points where $u$ cannot be tested. Fix $p_0\in \Omega$ arbitrarily. Since $u$ is locally bounded, for $\vep>0$ small, we can find $p_\vep\in \Omega$ in a neighborhood of $p_0$ such that 
\[
p\mapsto u(p)-\frac{1}{\vep} |p_0^{-1}\cdot p|_G^4
\]
attains a local maximum in $\Omega$ at $p_\vep$. In particular, we have 
\beq\label{lem sub eq1}
u(p_0)\leq u(p_\vep)-\frac{1}{\vep} |p_0^{-1}\cdot p_\vep|_G^4.
\eeq
By the local boundedness of $u$, we have $p_\vep\to p_0$ as $\vep\to 0$.
Noticing that $u$ is tested from above by a smooth function at $p_\vep$, we may apply our result shown in the first part to deduce 
$u(p_\vep)\leq f(p_\vep). $
It follows from \eqref{lem sub eq1} that $u(p_0)\leq f(p_\vep)$.
Sending $\vep\to 0$ and applying the upper semicontinuity of $f$, we end up with $u(p_0)\leq f(p_0)$.
\end{proof}

Let us present more properties for  \eqref{hj eq}.
\begin{lem}[Basic properties]\label{lem basic property}
Suppose that $\Omega$ is a domain in $\H$. For each locally bounded $f\in USC(\Omega)$, let $\A[f]$ denote the set of all subsolutions of \eqref{hj eq}.  Then the following properties hold. 
\begin{enumerate}
\item[(i)] (Monotonicity) For any locally bounded $f_1, f_2\in USC(\Omega)$ satisfying $f_1\leq f_2$ in $\Omega$, $\A[f_1]\subset \A[f_2]$ holds. 
\item[(ii)] (Constant addition invariance) For any $c\in \R$ and $u\in \A[f]$, $u+c\in \A[f+c]$ holds.
\item[(iii)] (Left translation invariance) For any $\eta\in \H$, $u_\eta\in \A[f_\eta]$ holds, where $u_\eta$ and $f_\eta$ are given by 
\[
u_\eta(p)=u(\eta\cdot p), \quad f_\eta(p)=f(\eta\cdot p), \quad p\in \Omega_\eta.
\]
\end{enumerate}
Analogous properties to the above also hold for supersolutions of \eqref{hj eq}. 
\end{lem}
We omit the details of the proof, since it is quite straightforward from the structure of the Hamiltonian.

\subsection{Comparison principle}\label{sec:comp}
We next establish a comparison principle for \eqref{hj eq} under the conditions \eqref{dirichlet} and \eqref{coercive-like}. General Dirichlet boundary problems will also be briefly discussed later.  

\begin{thm}[Comparison principle with constant boundary data]\label{thm comp}
Let $\Omega$ be a bounded domain in $\H$ and $f\in C(\Omega)$. Let $u\in USC(\Oba)$ and $v\in LSC(\Oba)$ be respectively a subsolution and a supersolution of \eqref{hj eq}. Assume in addition that 
\beq\label{comp add1}
u\leq v=K \quad \text{on $\partial\Omega$}
\eeq
and $u\leq K $ in $\Oba$ for some $K\in\R$. 
Then $u\leq v$ in $\Oba$. 
\end{thm}
\begin{proof}
Assume by contradiction that $\max_{\Oba} (u-v)=\sigma$ for some $\sigma>0$. For $\vep>0$ small, we consider 
\[
\Phi_\vep(p, q)=u(p)-v(q)-\frac{|p\cdot q^{-1}|_G^4}{\vep}
\]
for $p, q\in \Oba$. 
 It is clear that $\Phi_\vep$ attains a positive maximum in $\Oba\times \Oba$. Let $(p_\vep, q_\vep)\in \Oba\times \Oba$ be a maximizer. We thus obtain 
\beq\label{cp eq2}
\Phi_\vep(p_\vep, q_\vep)\geq \sigma,
\eeq
which implies that 
\beq\label{cp eq0}
u(p_\vep)-v(q_\vep)\geq \sigma
\eeq
and 
\beq\label{cp eq1}
\frac{|p_\vep\cdot q_\vep^{-1}|_G^4}{\vep}\leq u(p_\vep)-v(q_\vep)-\sigma
\eeq
for all $\vep>0$. Due to the boundedness of $u$ and $v$, it follows from \eqref{cp eq1} that $|p_\vep\cdot q_\vep^{-1}|_G\to 0$,
which in turn implies that $d_H(p_\vep, q_\vep)=|p_\vep^{-1}\cdot q_\vep|_G\to 0$ as $\vep\to 0$.

By taking a subsequence, still indexed by $\vep$, we have $p_\vep, q_\vep\to p_0$ as $\vep\to 0$ for some point $p_0\in \Oba$. Thanks to \eqref{cp eq2} and the assumption that $u\leq v$ on $\partial \Omega$, we deduce that $p_0\in \Omega$ and $p_\vep, q_\vep\in \Omega$ for all $\vep>0$ small. 

We now apply the definition of subsolutions and supersolutions. Note that $u-\varphi_1$ attains a maximum at $p_\vep$ and $v-\varphi_2$ attains a minimum at $q_\vep$, where we take
\[
\varphi_1(p)= v(q_\vep)+{|p\cdot q_\vep^{-1}|_G^4\over \vep},
\]
\[
\varphi_2(q)=u(p_\vep)-{|p_\vep\cdot q^{-1}|_G^4\over \vep}
\]
for $p, q\in \H$. Writing 
\[
p_\vep=\left(x_{p_\vep}, y_{p_\vep}, z_{p_\vep}\right),\quad q_\vep=\left(x_{q_\vep}, y_{q_\vep}, z_{q_\vep}\right),
\]
we see, by direct calculations, that 
\beq\label{cp new5}
\begin{aligned}
&\nabla_H \varphi_1(p_\vep)=\nabla_H \varphi_2 (q_\vep)\\
&={1\over \vep}\big(4A_\vep (x_{p_\vep}-x_{q_\vep})-16 B_\vep(y_{p_\vep}+y_{q_\vep}),\ 4A_\vep(y_{p_\vep}-y_{q_\vep})+16B_\vep(x_{p_\vep}+x_{q_\vep})\big),
\end{aligned}
\eeq
where 
\[
A_\vep=(x_{p_\vep}-x_{q_\vep})^2+(y_{p_\vep}-y_{q_\vep})^2,\quad 
B_\vep=z_{p_\vep}-z_{q_\vep}-{1\over 2}x_{p_\vep}y_{q_\vep}+{1\over 2}y_{p_\vep}x_{q_\vep}.
\]
We next discuss two different cases. 

Case 1. Suppose that there exists a subsequence, again indexed by $\vep$ for simplicity of notation, such that 
\[
\sup\{\la \nabla_H \varphi_2(q_\vep),  (q_\vep^{-1}\cdot \eta)_h\ra: \eta\in \hat{S}_{q_\vep}(v)\}=0.
\]
Then applying the definition of supersolutions to $v$, we get
\beq\label{cp eq3}
v(q_\vep)\geq f(q_\vep).
\eeq
On the other hand, by Lemma \ref{lem sub}, we obtain 
\beq\label{cp eq4}
u(p_\vep)\leq f(p_\vep).
\eeq
Combining \eqref{cp eq3} and \eqref{cp eq4}, we have 
\[
u(p_\vep)-v(q_\vep)\leq f(p_\vep)-f(q_\vep)
\]
for $\vep>0$ small. Letting $\vep\to 0$, we use the continuity of $f$ to get
\[
\limsup_{\vep\to 0} \left(u(p_\vep)-v(q_\vep)\right)\leq 0,
\]
which is a contradiction to \eqref{cp eq0}. 

Case 2. Suppose that 
\[
\sup\left\{\la \nabla_H \varphi_2(q_\vep),  (q_\vep^{-1}\cdot \eta)_h\ra: \eta\in \hat{S}_{q_\vep}(v)\right\}>0
\]
for all $\vep>0$ small. We then can find a sequence $\eta_{\vep, n}\in \hat{S}_{q_\vep}(v)$ such that 
\beq\label{cp new4}
\la \nabla_H \varphi_2(q_\vep),  (w_n)_h\ra\geq \sup\left\{\la \nabla_H \varphi_2(q_\vep),  (q_\vep^{-1}\cdot \eta)_h\ra: \eta\in \hat{S}_{q_\vep}(v)\right\}-{1\over n}
\eeq
as $n\to \infty$, where $w_n=q_\vep^{-1}\cdot \eta_{\vep, n}\in \H_0$.
In particular, we have $v(\eta_{\vep, n})\leq v(q_\vep)$.

In view of  \eqref{coercive-like} and \eqref{cp eq0}, we get $v(q_{\vep})\leq K-\sigma$.
On the other hand, since $v=K$ due to \eqref{comp add1}, we see that as $n\to \infty$ and  and $\vep\to 0$, $\eta_{\vep, n}$ cannot converge to a boundary point. In other words, there exists $r>0$ such that $B_r(\eta_{\vep, n})\subset \Omega$.
We may assume that $\vep$ is small enough so that $r>|p_\vep\cdot q_{\vep}^{-1} |_G$. 

Besides, it also follows from \eqref{cp new4} that
\beq\label{cp new1}
\la \nabla_H \varphi_2(q_\vep),  (w_n)_h\ra>0
\eeq
for any $n\geq 1$ large. 
Taking 
\[
\mu(s)=|p_\vep\cdot (s w_n)\cdot \eta_{\vep, n}^{-1} |_G^4=|p_\vep\cdot (s-1)w_n\cdot q_\vep^{-1}|_G^4, \quad s\in \R
\]
with $sw_n$ denoting the usual constant multiple of $w_n$ by the factor $s$, i.e.,  
\[
s w_n=(s x_{w_n}, s y_{w_n}, 0)\quad\text{for }\ w_n=(x_{w_n}, y_{w_n}, 0),
\]
we get by direct calculations
 \[
 \begin{aligned}
 \mu'(s)&=4A_{\vep, s}(x_{p_\vep}-x_{q_\vep}+(s-1) x_{w_n}) x_{w_n}+4A_{\vep, s}(y_{p_\vep}-y_{q_\vep}+(s-1) y_{w_n}) y_{w_n}\\
 &\quad -16B_{\vep, s} (y_{p_\vep}+y_{q_\vep})x_{w_n}+16B_{\vep, s}(x_{p_\vep}+x_{q_\vep})y_{w_n},
 \end{aligned}
 \]
where we let 
\[
A_{\vep, s}=(x_{p_\vep}-x_{q_\vep}+(s-1)x_{w_n})^2+(y_{p_\vep}-y_{q_\vep}+(s-1)y_{w_n})^2,
\]
\[
B_{\vep, s}=z_{p_\vep}-z_{q_\vep}-{1\over 2} x_{p_\vep} y_{q_\vep}+{1\over 2}y_{p_\vep}x_{q_\vep}-{1\over 2} (s-1) x_{w_n} y_{q_\vep}+{1\over 2} (s-1) y_{w_n} x_{q_\vep}.
\]
It is then clear that 
\[
\begin{aligned}
\mu'(1)&=\big(4A_\vep(x_{p_\vep}-y_{q_\vep})-16 B_\vep(y_{p_\vep}+y_{q_\vep})\big) x_{w_n}+\big(4A_\vep(y_{p_\vep}-y_{q_\vep})+16B_\vep(x_{p_\vep}+x_{q_\vep})\big)y_{w_n}\\
&= \vep\la \nabla_H \varphi_2(q_\vep), w_n\ra.
\end{aligned}
\]
Owing to \eqref{cp new1}, we are led to $\mu'(1)>0$, which implies 
\[
\mu(s)=|p_\vep\cdot (s w_n)\cdot \eta_{\vep, n}^{-1} |_G^4< |p_\vep\cdot q_{\vep}^{-1} |_G^4
\]
for any $s<1$ sufficiently close to $1$. This amounts to saying that we can take $s_n\in (0, 1)$ such that $s_n\to 1$ as $n\to \infty$ and 
\beq\label{cp new2}
|\xi_{\vep, n}\cdot \eta_{\vep, n}^{-1}|_G^4<|p_\vep\cdot q_{\vep}^{-1} |_G^4
\eeq
for $\xi_{\vep, n}=p_\vep\cdot (s_n w_n)$. This yields $\xi_{\vep, n}\in B_r(\eta_{\vep, n})$ and thus $\xi_{\vep, n}\in \Omega$. It is also clear that $\xi_{\vep, n}\in \H_{p_{\vep}}$ and 
\beq\label{cp new3}
\left|\left(p_\vep^{-1}\cdot \xi_{\vep, n}\right)_h -\left(q_\vep^{-1}\cdot \eta_{\vep, n}\right)_h\right|\to 0\quad \text{as $n\to \infty$}.
\eeq

In view of the maximality of $\Phi_\vep$ at $(p_\vep, q_\vep)$, we have 
\[
u(\xi_{\vep, n})-v(\eta_{\vep, n})-{1\over \vep} |\xi_{\vep, n}\cdot \eta_{\vep, n}^{-1}|_G^4\leq u(p_\vep)-v(q_\vep)-{1\over \vep}|p_\vep\cdot q_\vep^{-1} |_G^4
\]
which by \eqref{cp new2} yields 
\[
u(\xi_{\vep, n})-u(p_{\vep})< v(\eta_{\vep, n})-v(q_\vep)\leq 0.
\]
We have shown that $\xi_{\vep, n}\in S_{p_\vep}(u)$. Adopting the definition of subsolutions, we deduce that 
\[
u(p_\vep)+\la \nabla_H\varphi_1(p_\vep),\ \left(p_\vep^{-1}\cdot  \xi_{\vep, n}\right)_h\ra\leq f(p_\vep).
\]
On the other hand,  we can also apply the definition of supersolutions, together with \eqref{cp new4}, to get
\beq\label{general comp2}
v(q_\vep)+\la \nabla_H\varphi_2(q_\vep),\ \left(q_\vep^{-1}\cdot  \eta_{\vep, n}\right)_h\ra\geq f(q_\vep)-{1\over n}
\eeq
for $n\geq 1$ large. 
Combining these two inequalities, we use \eqref{cp new5} to obtain
\[
u(p_\vep)-v(q_\vep)\leq f(p_\vep)-f(q_\vep)-\la \nabla_H\varphi_1(p_\vep), \left(p_\vep^{-1}\cdot \xi_{\vep, n}\right)_h -\left(q_\vep^{-1}\cdot \eta_{\vep, n}\right)_h \ra+{1\over n}.
\]
Sending $n\to \infty$ and applying \eqref{cp new3}, we get
\[
u(p_\vep)-v(q_\vep)\leq f(p_\vep)-f(q_\vep),
\]
which is a contradiction to \eqref{cp eq0} and the continuity of $f$. 
\end{proof}

It is possible to give a slightly different comparison theorem for more general Dirichlet boundary problems under the h-convexity assumption on $\Omega$. To this end, we prove the following lemma using the intrinsic cone property of h-convexity given in \cite[Theorem 1.4]{ArCaMo}; see also \cite{MoM} for regularity results related to this property. 
\begin{lem}\label{rmk convex boundary}
Let $\Omega \subset \H$ be an open set. If $\Omega$ is h-convex, then for every $p\in \partial \Omega$ and $q\in \H_p\cap \Omega$, the horizontal segment $(p, q]$ joining $p$ and $q$ stays in $\Omega$. 
\end{lem}
\begin{proof}
Assume by contradiction that there exist $p \in \partial \Omega$, $q \in \H_p \cap \Omega$ such that the half-open horizontal segment $(p, q]$ does not stay in $\Omega$. Let $w$ be the closest point to $q$ on $[q,p)$ that is not in $\Omega$, i.e., $w=p\cdot \lambda_0(p^{-1}\cdot q)$, where
\[
\lambda_0=\sup\{\lambda\geq 0: p\cdot \lambda(p^{-1}\cdot q)\notin \Omega\}.
\]
 Then $w \in \partial \Omega$, $q\in \H_w$ and $w\cdot \lambda(w^{-1}\cdot q)\in \Omega$ for all $\lambda\in (\lambda_0, 1)$. Such  $w$ is a so-called non-characteristic point defined in \cite{ArCaMo}. In view of \cite[Theorem 1.4]{ArCaMo}, the h-convexity of $\Omega$ implies the existence of an intrinsic cone in the exterior of $\Omega$ with vertex $w$ and axis along the horizontal segment $(p, w]$, which further enables us to find a point $z \in (p, w)$ and $r>0$ such that $B_r(z) \subset \H \setminus \Omega$. Noticing that there exists a sequence $p_j \in \Omega$ such that $p_j \to p$ as $j \to \infty$, we can take $q_j=p_j \cdot p^{-1}\cdot q$ and $z_j = p_j \cdot p^{-1} \cdot z$ so that $q_j, z_j \in \H_{p_j}$ and $q_j\to q$, $z_j\to z$ as $j\to \infty$. 
 We choose $j>0$ large such that $q_j\in \Omega$ and $z_j\in B_r(z)$. This contradicts the h-convexity of $\Omega$, since $p_j, q_j \in \Omega$, but the point $z_j$ on the horizontal segment $[p_j, q_j]$ is not in $\Omega$.
\end{proof} 

\begin{thm}[Comparison principle under domain convexity]\label{thm comp2}
Let $\Omega$ be a bounded h-convex domain in $\H$ and $f\in C(\Omega)$. Let $u\in USC(\Oba)$ and $v\in LSC(\Oba)$ be respectively a subsolution and a supersolution of \eqref{hj eq}. 
If $u\leq v$ on $\partial \Omega$, then $u\leq v$ in $\Oba$. 
\end{thm}
\begin{proof}
The proof is almost the same as that of Theorem \ref{thm comp} except for some necessary modifications to handle this general  boundary condition. The only difference lies at the argument by contradiction in Case 2 on the occasion when $\eta_{\vep, n}\in \hat{S}_{q_\vep}$ converges to a boundary point $\eta_0\in \partial \Omega$ as $n\to \infty$ and $\vep\to 0$. (This gives a contradiction to the conditions \eqref{coercive-like} and \eqref{comp add1} in the setting of Theorem \ref{thm comp}.) 

Let us derive a contradiction only in this case. Thanks to the condition $u\leq v$ on $\partial \Omega$, we have $u(\eta_0)\leq v(\eta_0)$, which by \eqref{cp eq0} and the upper semicontinuity of $u$ yields the existence of $r>0$ such that 
\beq\label{general comp1}
B_r(\eta_0)\cap \Omega\subset \{ z \in \Omega: u(z) < u(p_\vep)\}
\eeq
for any $\vep>0$ small. 

Since $p_\vep, q_\vep\to p_0\in \Omega$, by Lemma \ref{rmk convex boundary} we can utilize the h-convexity of $\Omega$ to deduce that
\[
\eta^\lambda=(1-\lambda)p_0+\lambda \eta_0\in \Omega
\]
for all $0\leq \lambda<1$. We may also choose $\lambda$ close to $1$, depending only on $r$, such that $\eta^\lambda\in B_r(\eta_0)\cap\Omega$. Letting 
\[
\eta^\lambda_{\vep, n}=(1-\lambda)q_\vep+\lambda \eta_{\vep, n}=q_\vep\cdot (\lambda w_n),
\]
we have $\eta^\lambda_{\vep, n}\in B_r(\eta_0)\cap \Omega$ when $n$ is sufficiently large and $\vep$ is sufficiently small. Let us now take $\xi^\lambda_{\vep, n}=p_\vep\cdot (\lambda w_n)$. It is easily seen that 
\[
|\xi^\lambda_{\vep, n}\cdot  (\eta^\lambda_{\vep, n})^{-1}|_G\to 0 \quad \text{as $n\to \infty$ and $\vep\to 0$},
\]
and thus $\xi^\lambda_{\vep, n}\in B_r(\eta_0)$ for $n$ large and $\vep$ small. By \eqref{general comp1}, we have $\xi^\lambda_{\vep, n}\in S_{p_\vep} (u)$. It follows from the definition of subsolutions that 
\beq\label{general comp3}
u(p_\vep)+\la \nabla_H\varphi_1(p_\vep),\ \lambda(w_n)_h\ra\leq f(p_\vep).
\eeq
Combining this with \eqref{general comp2} and using \eqref{cp eq0} and \eqref{cp new5}, we obtain
\[
\sigma \leq u(p_\vep)-v(q_\vep)\leq {1\over n}+(1-\lambda)\la \nabla_H\varphi_1(p_\vep),\ (w_n)_h\ra.
\]
Applying \eqref{general comp3} and \eqref{cp eq0} again, we are led to 
\[
\sigma\leq  {1\over n}+{1-\lambda \over \lambda}\left(f(p_\vep)-u(p_\vep)\right)\leq {1\over n}+{1-\lambda \over \lambda}\left(f(p_\vep)-v(p_\vep)\right)
\] 
Sending $n\to \infty$ and $\vep\to 0$, we get
\[
\sigma\leq {1-\lambda \over \lambda}\left(f(p_0)-v(p_0)\right).
\]
Passing to the limit as $\lambda\to 1$ immediately yields a contradiction. 
\end{proof}

\subsection{Existence of solutions}\label{sec:exist}

We also adapt Perron's method to show the existence of solutions of \eqref{hj eq} with $f$ satisfying \eqref{dirichlet-f} and \eqref{coercive-f}. In the sequel, for a locally bounded function $u:\Omega \to \mathbb{R}$, we denote by $u^*$ and $u_*$ its upper and lower semicontinuous envelopes respectively.

\begin{thm}[Existence of solutions]\label{thm existence}
Let $\Omega\subset \H$ be a bounded domain. Assume that $f\in C(\Oba)$ satisfies 
\eqref{dirichlet-f} and \eqref{coercive-f} for some $K\in \R$. 
Assume that there exists a subsolution $\underline{u}\in C(\Oba)$ of \eqref{hj eq} satisfying $\underline{u}=K$ on $\partial \Omega$. For any $p\in \Oba$, let 
\beq\label{perron eq}
U(p)=\sup\{u(p): \text{$u\in USC(\Oba)$ is a subsolution of \eqref{hj eq}}\}. 
\eeq
Then $U^*$ is continuous in $\Oba$ and is the unique solution of \eqref{hj eq} satisfying $U^*\leq f\leq K$ in $\Oba$ and $U^*=f=K$ on $\partial \Omega$. 
\end{thm}

\begin{rmk}
In the theorem above, if the domain $\Omega$ is further assumed to be h-convex, then the existence of $\underline{u}$ is guaranteed; see Proposition \ref{prop lower} below for an explicit construction of $\underline{u}$ satisfying the required conditions. 
\end{rmk}

\begin{rmk}
In view of Lemma \ref{lem sub}, we see that any subsolution $u$ satisfies $u\leq f$ in $\Omega$. Therefore $\underline{u}\leq u\leq f$ in $\Oba$ holds and $U$ in \eqref{perron eq} is well-defined. 
\end{rmk}

Let us first show that $U$ is a subsolution. To this purpose, we prove the following result, where we do not need the coercivity-like conditions \eqref{dirichlet} and \eqref{coercive-like}.  

\begin{prop}[Maximum subsolution]\label{prop sub}
Suppose that $\Omega$ is a domain in $\H$ and $f\in USC(\Omega)$. Let $\A$ be a family of subsolutions of \eqref{hj eq}. 
 Let $v$ be given by 
\[
v(p)=\sup\{u(p): \text{$u\in \A$}\}, \quad p\in \Omega. 
\]
Then $v^\ast$ is a subsolution of \eqref{hj eq}. 
\end{prop}
\begin{proof}
Suppose that there exist $\varphi\in C^1(\Omega)$ and $p_0\in \Omega$ such that $v^\ast-\varphi$ attains a strict maximum at $p_0$. Then we can find a sequence $u_j\in \A$ and $p_j\in \Omega$ such that $u_j-\varphi$ attains a maximum at $p_j$, and, as $j\to \infty$,  
\beq\label{exist sub1}
p_j\to p_0, \quad u_j(p_j)\to v^\ast(p_0).
\eeq
 
Let us first consider the case when $\nabla_H\varphi(p_0)\neq 0$. It follows immediately that $S_{p_0}(v^\ast)\neq \emptyset$. For any $\delta>0$ small, there exists $q_\delta\in S_{p_0}(v^\ast)$ such that 
\beq\label{exist sub2}
\sup\left\{\la \nabla_H \varphi(p_0),\  (p_0^{-1}\cdot \xi)_h\ra:\ \xi\in S_{p_0}(v_\ast) \right\}\leq \la \nabla_H \varphi(p_0),\  (p_0^{-1}\cdot q_\delta)_h\ra+\delta. 
\eeq
Let us take $\xi_j\in \H_{p_j}\cap \Omega$ such that $\xi_j\to q_\delta$ as $j\to \infty$.
Since $v^\ast(q_\delta)<v^\ast(p_0)$ and 
\[
\limsup_{j\to \infty}u_j(\xi_j)\leq \limsup_{j\to \infty} v^\ast(\xi_j)\leq v^\ast(q_\delta),
\]
we get $u_j(\xi_j)<v^\ast(p_0)$ and thus $\xi_j\in S_{p_j}(u_j)$ for $j\geq 1$ sufficiently large.

 Applying Definition \ref{def sub}, we have 
\[
u_j(p_j)+\sup\left\{\la \nabla_H \varphi(p_j),\  (p_j^{-1}\cdot \xi)_h\ra:\ \xi\in S_{p_j}(u_j) \right\}\leq f(p_j),
\]
which implies 
\[
u_j(p_j)+\la \nabla_H \varphi(p_j),\  (p_j^{-1}\cdot \xi_j)_h\ra\leq f(p_j).
\]
Letting $j\to \infty$, by \eqref{exist sub1} and the upper semicontinuity of $f$, we are led to 
\[
v^\ast(p_0)+\la \nabla_H \varphi(p_0),\  (p_0^{-1}\cdot q_\delta)_h\ra\leq f(p_0).
\]
It follows from \eqref{exist sub2} that 
\[
v^\ast(p_0)+\sup\left\{\la \nabla_H \varphi(p_0),\  (p_0^{-1}\cdot \xi)_h\ra:\ \xi\in S_{p_0}(v^\ast) \right\}\leq f(p_0)+\delta. 
\]
Due to the arbitrariness of $\delta>0$, we obtain 
\[
v^\ast(p_0)+\sup\left\{\la \nabla_H \varphi(p_0),\  (p_0^{-1}\cdot \xi)_h\ra:\ \xi\in S_{p_0}(v^\ast) \right\}\leq f(p_0).
\]

Let us discuss the case when $\nabla_H \varphi(x_0)=0$. 
Since $u_j-\varphi$ attains a maximum at $p_j$,  by Lemma \ref{lem sub}, we get $u_j(p_j)\leq f(p_j)$. 
Sending $j\to \infty$, we have $v^\ast(p_0)\leq f(p_0)$, as desired. 
Our proof is now complete. 
\end{proof}
\begin{rmk}
Proposition \ref{prop sub} can be used in more general circumstances than the setting in Theorem \ref{thm existence}. Note that $f$ is only assumed to be upper semicontinuous in $\Omega$ in contrast to the continuity assumption in Theorem \ref{thm existence}.
\end{rmk}
\begin{rmk}\label{rmk supremum}
The proof of Proposition \ref{prop sub}, which is just an adaptation of Perron's argument, is certainly not restricted to \eqref{hj eq}. It can be extended with ease to a more general class of nonlocal Hamilton-Jacobi equation including 
\[
\sup\left\{\la \nabla_H u(p),  (p^{-1}\cdot \xi)_h\ra:  \xi\in \hat{S}_p(u)\right\}=0.
\]
In this case, in view of Theorem \ref{thm char}, our result simply means that the upper semicontinuous envelope of the pointwise supremum of a class of upper semicontinuous h-quasiconvex functions is also h-quasiconvex. 
\end{rmk}

The following result shows that the maximal subsolution is a solution.

\begin{prop}[Locally greater subsolutions]\label{prop super}
Let $\Omega\subset \H$ be a bounded domain and $f\in LSC(\Omega)$. Let $u$ be a subsolution of \eqref{hj eq}. Assume that $u_\ast\geq f$ on $\pO$. If $u_\ast$ fails to satisfy the supersolution property at $p_0\in \Omega$, then there exist $r>0$ and a subsolution $U_r$ of \eqref{hj eq} such that
\beq\label{greater sub1}
\sup_{B_r(p_0)} (U_r-u)>0
\eeq
and  
\beq\label{greater sub2}
U_r=u \quad \text{in $\Omega\setminus B_r(p_0)$.} 
\eeq
\end{prop}
\begin{proof}
Since $u_\ast$ fails to satisfy the supersolution property at $p_0$, there exists $\varphi\in C^1(\Omega)$ such that $u_\ast-\varphi$ attains a  minimum at $p_0$ but 
\beq\label{eq greater1}
u_\ast(p_0)+\sup\left\{\la \nabla_H \varphi(p_0),  (p_0^{-1}\cdot \xi)_h\ra:\ \xi\in \hat{S}_{p_0}(u_\ast) \right\}< f(p_0).
\eeq
Since the supremum is nonnegative, it is then clear that 
\beq\label{eq greater3}
u_\ast(p_0)<f(p_0).
\eeq

By adding $a |p\cdot p_0^{-1}|_G^4+b$ to $\varphi$ with $a<0$ and $b\in \R$, we may additionally assume that there exists $r>0$ small satisfying
\[
(u_\ast-\varphi)(p)>(u_\ast-\varphi)(p_0)=0
\]
for all $p\in B_r(p_0)\setminus \{p_0\}$. 
We also take $\delta>0$ accordingly small such that 
\[
\varphi(p)+\delta\leq u(p) \quad \text{for $p\in B_r(p_0)\setminus B_{r/2}(p_0)$.}
\]
Letting 
\[
U_r(p)=\begin{cases}
\max\{\varphi(p)+\delta, u(p)\} & \text{if $p\in B_r(p_0)$,}\\
u(p) & \text{if $p\notin B_r(p_0)$,}
\end{cases}
\]
we see that \eqref{greater sub1} and \eqref{greater sub2} hold. 
In what follows we show that $U_r$ is a subsolution of \eqref{hj eq} for such $r>0$ and $\delta>0$ small. 

Suppose that $\psi\in C^1(\Omega)$ is a test function such that $U_r-\psi$ attains a maximum at some $q_0\in \Omega$. We divide our argument into two cases. 

Case 1. Suppose that $U_r(q_0)=u(q_0)$. It follows that $u-\psi$ attains a maximum at $q_0$. Since $u$ is a subsolution, we have 
\beq\label{eq greater2}
u(q_0)+\sup\{\la \nabla_H \psi(q_0), (q_0^{-1}\cdot \eta)_h\ra: \eta\in S_{q_0}(u) \}\leq f(q_0). 
\eeq
Noticing that $U_r\geq u$ in $\Omega$, we have $S_{q_0}(U_r)\subset S_{q_0}(u)$, which by \eqref{eq greater2} yields
\[
U_r(q_0)+\sup\{\la \nabla_H \psi(q_0), (q_0^{-1}\cdot \eta)_h\ra: \eta\in S_{q_0}(U_r) \}\leq f(q_0),
\]
as desired. 

Case 2. Suppose that $U_r(q_0)=\varphi(q_0)+\delta$. In this case, we have $q_0\in B_r(p_0)$. Also, it is easily seen that $\varphi+\delta-\psi$ attains a maximum at $q_0$, which yields
\[
\nabla_H \psi(q_0)=\nabla_H \varphi(q_0). 
\]
It thus suffices to show that there exists $r, \delta>0$ small such that 
\beq
U_r(q)+\sup\{\la\nabla_H\varphi(q), (q^{-1}\cdot \eta)_h\ra: \eta\in S_q(U_r)\}\leq f(q)
\eeq
for all $q\in B_r(p_0)$ satisfying $U_r(q)=\varphi(q)+\delta$. 

Arguing by contradiction, we take $r_n, \delta_n>0$ and $q_n\in B_{r_n}(p_0)$ such that $r_n, \delta_n\to 0$ as $n\to \infty$, 
\[
U_{r_n}(q_n)=\varphi(q_n)+\delta_n,
\]
 and
\beq\label{eq greater4}
\varphi(q_n)+2\delta_n+\sup\{\la\nabla_H\varphi(q_n), (q_n^{-1}\cdot \eta)_h\ra: \eta\in S_{q_n}(U_r)\}> f(q_n).
\eeq
Note that if there exists a subsequence of $q_n$ at which $\nabla_H\varphi(q_n)=0$, then sending $n\to \infty$ in \eqref{eq greater4} yields $\varphi(p_0)\geq f(p_0)$, which is clearly a contradiction to \eqref{eq greater3}.

We thus only need to consider the case when $\nabla_H\varphi(q_n)\neq 0$  for all $n\geq 1$. By \eqref{eq greater4}, this implies the existence of $\eta_n\in S_{q_n}(U_{r_n})$ such that 
\beq\label{eq greater6}
\varphi(q_n)+2\delta_n+\la\nabla_H\varphi(q_n), (q_n^{-1}\cdot \eta_n)_h\ra> f(q_n).
\eeq
In particular, we have $\eta_n\in \H_{q_n}\cap \Omega$ and 
\beq\label{eq greater5}
u(\eta_n)\leq U_{r_n}(\eta_n)< \varphi(q_n)+\delta_n.
\eeq
Due to the boundedness of $\Omega$, we may find a subsequence, still indexed by $n$, such that $\eta_n\to \eta_0$ for some $\eta_0\in \Oba$. Since the horizontal plane $\H_p$ is continuous in $p$, we have $\eta_0\in \H_{p_0}$. Moreover, passing to the limit in \eqref{eq greater5}, we obtain
\[
u_\ast(\eta_0)\leq \varphi(p_0)=u_\ast(p_0),
\]
which by \eqref{eq greater3} implies $u_\ast(\eta_0)<f(p_0)$. 
In view of the assumption that $u_\ast\geq f$ on $\partial \Omega$, we immediately have $\eta_0\in \Omega$. In other words, we have shown that 
\beq\label{eq greater7}
\eta_0\in \hat{S}_{p_0}(u_\ast).
\eeq

Finally, taking $\liminf_{n\to \infty}$ in \eqref{eq greater6}, we are led to
\[
u_\ast(p_0)+\la \nabla_H\varphi(p_0), (p_0^{-1}\cdot \eta_0)_h\ra\geq f(p_0),
\]
which, together with \eqref{eq greater7}, yields a contradiction to \eqref{eq greater1}.
\end{proof}

Let us complete the proof of Theorem \ref{thm existence}.
\begin{proof}[Proof of Theorem \ref{thm existence}]
Let $U$ be given by \eqref{perron eq}. By Proposition \ref{prop sub}, $U^\ast$ is a subsolution of \eqref{hj eq}. Also, by Lemma \ref{lem sub} and the continuity of $f$ in $\Oba$, we obtain 
\beq\label{existence pf1}
U^\ast\leq f\leq K \quad \text{ in $\Oba$.}
\eeq
By assumptions on $\underline{u}$, we also have $U\geq \underline{u}$ in $\Oba$ and 
\beq\label{existence pf2}
(U^\ast)_\ast \geq U_\ast\geq \underline{u}=K\quad\text{ on $\partial \Omega$.}
\eeq
Applying Proposition \ref{prop super}, we see that $(U^\ast)_\ast$ is a supersolution of \eqref{hj eq}, for otherwise we can construct a  subsolution locally larger than $U^\ast$, which contradicts the definition of $U$. Noticing that 
\eqref{existence pf1} and \eqref{existence pf2} are combined to imply that 
\[
U^\ast=(U^\ast)_\ast=K\quad \text{on $\pO$,}
\]
by the comparison principle, Theorem \ref{thm comp}, we get $U^\ast\leq (U^\ast)_\ast$ in $\Oba$. It follows that $U^\ast$ is continuous in $\Oba$ and is the unique solution of \eqref{hj eq} satisfying $U^\ast=f=K$ on $\partial\Omega$.
\end{proof}

Following the proof above, one can obtain an existence result for general Dirichlet boundary problems under the h-convexity of $\Omega$.
\begin{thm}[Existence for general Dirichlet problems]\label{thm existence2}
Let $\Omega\subset \H$ be a bounded h-convex domain and $f\in C(\Oba)$. 
Assume that there exist a subsolution $\underline{u}\in C(\Oba)$ of \eqref{hj eq} such that $\underline{u}=f$ on $\partial \Omega$. Let $U$ be given by \eqref{perron eq}. 
Then $U^*$ is continuous in $\Oba$ and is the unique solution of \eqref{hj eq} satisfying $U^*=f$ on $\partial \Omega$. 
\end{thm}
We omit the detailed proof but remark that Theorem \ref{thm comp2} enables us to handle general boundary data following similar arguments above.

\section{H-quasiconvex envelope via PDE-based iteration}\label{sec:iteration}

An iteration is introduced in Section \ref{sec:intro} to find the h-quasiconvex envelope $Q(f)$ of a given function $f$ in a bounded h-convex domain $\Omega\subset\H$ with $f$ satisfying \eqref{dirichlet-f} and \eqref{coercive-f}. In this section, we first prove our main result, Theorem \ref{thm scheme bdry}. We shall later study a more general case when $\Omega$ is possibly unbounded without boundary data prescribed on $\pO$. 

We begin with an easy construction of an h-quasiconvex $\ul{f}$ as a lower bound of the whole scheme.
\begin{prop}[Existence of h-quasiconvex barriers]\label{prop lower}
Let $\Omega\subset \H$ be a bounded h-convex domain. Assume that $f\in C(\Oba)$ satisfies \eqref{dirichlet-f} and \eqref{coercive-f} for some $K\in \R$. Then there exists $\ul{f}\in C(\Oba)$ h-quasiconvex in $\Omega$ such that $\ul{f}\leq f$ in $\Oba$ and $\ul{f}=K$ on $\partial \Omega$.
\end{prop}
\begin{proof}
By Proposition \ref{prop metric-quasi}, we see 
that 
\[
\psi:=-\tilde{d}_H(\cdot, \H\setminus \Omega)
\]
is an h-quasiconvex function in $\H$. Since $f$ is continuous and $\Omega$ is bounded, there exists a modulus of continuity $\omega_f$, strictly increasing, such that
\[
f(p)\geq -\omega_f (\tilde{d}_H(p, \H\setminus \Omega))+K \quad \text{for all $p\in \Oba$.}
\]
Taking 
\beq\label{g-fun}
g(s)=-\omega_f(-s)+K , \quad s\le 0
\eeq
we immediately get 
\[
f(p)\geq g(-\tilde{d}_H(p, \H\setminus E))\geq g(\psi(p))\quad \text{for all $p\in \Omega$.}
\]
Let $\ul{f}=g\circ\psi$. We easily see that $\ul{f}\leq f$ in $\Oba$ and $\ul{f}=K$ on $\partial \Omega$.
Noticing that $g$ is continuous and nondecreasing, we deduce by Lemma \ref{lem geometric} that $\ul{f}$ is h-quasiconvex in $\Omega$. 
\end{proof}

Note that the existence of the h-quasiconvex envelope $Q(f)$ is guaranteed by the existence of $\ul{f}$. Moreover, due to the conditions above, we have $Q(f)=f=K$ on $\partial \Omega$. 

Let $u_0=f$. By Theorem \ref{thm existence},  for $n=1, 2, \ldots$ we can find a unique solution $u_n$ of \eqref{iteration eq} satisfying \eqref{data bdry} and \eqref{data bdry ineq}. 
By Lemma \ref{lem sub},  one can also see that 
\beq\label{monotone bdry}
\ul{f}\leq \ldots \leq u_n\leq u_{n-1}\leq \ldots \leq u_0 = f \quad\text{in $\Omega$ for $n=1, 2, \ldots$}.
\eeq

We proceed to prove Theorem \ref{thm scheme bdry}. To this end, we mention a fundamental fact on monotone sequences of functions. 
\begin{rmk}\label{rmk limits}
Since $u_n$ is non-increasing in $n$, the pointwise limit $\lim_{n\to \infty}u_n$ is equal to  $\limsup_{n\to \infty}^\ast u_n$, defined by
\[
\limsups_{n\to \infty} u_n(p)=\lim_{k\to \infty} \sup\left\{u_n(q): q\in B_{{1\over k}}(p), \ n\geq k\right\},\quad \text{$p\in \Omega$}.
\]
In fact, for any fixed $p\in \Omega$, $\vep>0$ and any $n\geq 1$, by the upper semicontinuity of $u_n$, we can take $k\geq 1$ sufficiently large to get
\[
u_n(p)\geq \sup\left\{u_n(q): q\in B_{{1\over k}}(p)\right\}-\vep.
\]
By the monotonicity of $u_n$ in $n$, we have 
\[
u_n(p)\geq \sup\{u_n(q): q\in B_{{1\over k}}(p), n\geq k\}-\vep,
\]
which yields
\[
u_n(p)\geq \limsups_{n\to \infty} u_n(p)-\vep.
\]
Letting $n\to \infty$ and $\vep\to 0$, we are led to 
\[
\lim_{n\to \infty} u_n(p)\geq \limsups_{n\to \infty} u_n(p).
\]
Since the reverse inequality clearly holds, we thus obtain the equality. 
\end{rmk}

\begin{prop}[Stability of h-quasiconvexity]\label{prop sub-stability}
Suppose that $\Omega$ is a bounded h-convex domain in $\H$. Let $u_0=f\in USC(\Omega)$ and $u_n\in USC(\Omega)$ be a subsolution of \eqref{iteration eq}. Then 
\[
u=\limsups_{n\to \infty} u_n
\]
is h-quasiconvex in $\Omega$. 
\end{prop}

\begin{proof}
Our argument below in based on the characterization given in Theorem \ref{thm char}. 
Suppose that there exist $p_0\in \Omega$ and $\varphi\in C^1(\Omega)$ such that $u-\varphi$ attains a strict maximum in $\Omega$ at $p_0$.  Then by Remark \ref{rmk limits}, there exists $p_n\in \Omega$ such that $u_n-\varphi$ attains a local maximum at $p_n$ and $p_n\to p_0$, $u_n(p_n)\to u(p_0)$ as $n\to \infty$. 

We may assume that $\nabla_H\varphi(p_0)\neq 0$, for otherwise it is clear that 
\[
\la\nabla_H \varphi(p_0), (p_0^{-1}\cdot \xi)_h\ra=0
\]
holds for all $\xi\in S_{p_0}(u)$.  We thus have $S_{p_0}(u)\neq \emptyset$. For any $\vep>0$, let $\xi_0\in S_{p_0}(u)$ satisfy 
\beq\label{scheme pf eq1}
\la \nabla_H\varphi(p_0), (p_0^{-1}\cdot \xi_0)_h\ra\geq \sup\left\{\la \nabla_H \varphi(p_0),  (p_0^{-1}\cdot \xi)_h\ra: \xi\in S_{p_0}(u)\right\}-\vep.
\eeq
Since $u(\xi_0)<u(p_0)$, we have $u_n(\xi_0)<u_n(p_n)$ when $n$ is sufficiently large. Then there is a point $\xi_n\in \H_{p_n}$ satisfying $u_n(\xi_n)<u_n(p_n)$ for $n$ large and $p_n^{-1}\cdot \xi_n=p_0^{-1}\cdot \xi_0$. 
Applying the definition of subsolutions of \eqref{iteration eq}, we get
\[
u_n(p_n)+\la \nabla_H\varphi(p_n), (p_n^{-1}\cdot \xi_n)_h\ra\leq u_{n-1}(p_n),
\]
which is equivalent to 
\[
u_n(p_n)+\la \nabla_H\varphi(p_n), (p_0^{-1}\cdot \xi_0)_h\ra\leq u_{n-1}(p_n).
\]
Passing to the limit as $n\to \infty$, we have
\[
u(p_0)+\la \nabla_H\varphi(p_0), (p_0^{-1}\cdot \xi_0)_h\ra\leq u(p_0).
\]
It follows from \eqref{scheme pf eq1} that
\[
\sup\left\{\la \nabla_H \varphi(p_0),  (p_0^{-1}\cdot \xi)_h\ra: \xi\in S_{p_0}(u)\right\}\leq \vep. 
\]
Letting $\vep\to 0$, we get
\[
\sup\left\{\la \nabla_H \varphi(p_0),  (p_0^{-1}\cdot \xi)_h\ra: \xi\in S_{p_0}(u)\right\}\leq 0.
\]
The proof for h-quasiconvexity of $u$ is now complete. 
\end{proof}

We are now in a position to prove Theorem \ref{thm scheme bdry}.
\begin{proof}[Proof of Theorem \ref{thm scheme bdry}]
In view of Proposition \ref{prop lower}, there exists $\ul{f}\in C(\Oba)$ h-quasiconvex such that $\ul{f}\leq f$ in $\Oba$ and $\ul{f}=f=K$ on $\partial \Omega$. 
Since $u_n\in C(\Oba)$ is a monotone sequence and $u_n\geq \ul{f}$ in $\Oba$, it converges to $u\in USC(\Oba)$ pointwise as $n\to \infty$.  We thus get $u= Q(f)=f=K$ on $\partial \Omega$. By Proposition \ref{prop sub-stability}, 
we see that $u$ is h-quasiconvex in $\Omega$. It follows immediately that $u\leq Q(f)$ in $\Omega$. 

Note that $Q(f)$ is a subsolution of \eqref{iteration eq} for every $n$ satisfying $Q(f)= f$ on $\partial\Omega$. Then by the comparison principle, we have $Q(f)\leq u_n$ in $\Oba$. We thus have $Q(f)\leq u$ in $\Oba$ and therefore $u_n\to Q(f)$ pointwise in $\Oba$ as $n\to \infty$. 

The uniform convergence of $u_n$ requires extra work. Let us extend the definitions of $u_n$ ($n=0, 1, 2, \ldots $) by constant $K$ to the whole space $\H$, that is, 
\[
u^K_n(p)=\begin{cases}
u_n(p) & \text{if $p\in \Omega$,}\\
K & \text{if $p\in\H\setminus\Omega$.}
\end{cases}
\]
Let us verify that $u^K_n$ is a solution of 
\beq\label{extend pde}
u^K_n(p)+H(p, u^K_n(p), \nabla_H u^K_n(p))=u^K_{n-1}(p) \quad \text{in $\H$}.
\eeq
In fact, it is obvious that $u^K_n$ is a supersolution. Below we show that it is a subsolution. 
Suppose that there is a test function $\varphi\in C^1(\H)$ such that $u^K_n-\varphi$ attains a maximum at $p_0\in \H$. We easily see that 
\beq\label{extend sub}
u^K_n(p_0)+\sup\left\{\la \nabla_H \varphi(p_0),  (p_0^{-1}\cdot \xi)_h\ra:\ \xi\in S_{p_0}(u_n^K) \right\}\leq u^K_{n-1}(p_0)
\eeq
 holds if $p_0\in \H\setminus \Oba$ or $p_0\in \Omega$. In the case that $p_0\in \partial \Omega$, extending $\ul{f}$ to get 
\[
\ul{f}^K(p)=\begin{cases}
\ul{f}(p) & \text{if $p\in \Omega$,}\\
K & \text{if $p\in\H\setminus\Omega$.}
\end{cases}
\]
we observe that $\ul{f}^K$ is h-quasiconvex in $\H$ and $\ul{f}^K-\varphi$ also attains a maximum at $p_0$; the latter is due to the facts that $\ul{f}^K\leq u^K_n$ in $\H$ and $\ul{f}^K=u^K_n=K$ on $\partial\Omega$. It follows from Theorem \ref{thm char} that 
\[
\sup\left\{\la \nabla_H \varphi(p_0),  (p_0^{-1}\cdot \xi)_h\ra:\ \xi\in S_{p_0}\left(\ul{f}^K\right) \right\}\leq 0,
\]
which implies that 
\[
\sup\left\{\la \nabla_H \varphi(p_0),  (p_0^{-1}\cdot \xi)_h\ra:\ \xi\in S_{p_0}\left(u_n^K\right) \right\}\leq 0.
\]
It is then clear that \eqref{extend sub} holds again, since $p_0\in \partial \Omega$ and thus $u_n^K(p_0)=u_{n-1}^K(p_0)=K$. Hence, $u_n^K$ is a subsolution of \eqref{extend pde}.

We next translate the extended solutions. For any fixed $h\in \H$, we have
\beq\label{equiconti1}
u^K_{n-1, h}(p)-\omega_f(|h|_G)\leq u^K_{n-1}(p)\quad \text{in $\H$ }
\eeq
holds for $n=1$, where $\omega_f$ is the modulus of continuity of $f$ in $\Oba$. 
Using Lemma \ref{lem basic property}, we can also deduce that $u^K_{n, h}(p)=u^K_n(h\cdot p)-\omega_f(|h|_G)$ is a subsolution of 
\[
u^K_{n, h}(p)+H(p, u^K_{n, h}(p), \nabla_H u^K_{n, h}(p))=u^K_{n-1, h}(p)-\omega_f(|h|_G) \quad \text{in $\H$,}
\]
and therefore by \eqref{equiconti1} is a subsolution of 
\[
u^K_{n, h}(p)+H(p, u^K_{n, h}(p), \nabla_H u^K_{n, h}(p))=u^K_{n-1}(p) \quad \text{in $\H$}
\]
in the case that $n=1$. Take a bounded h-convex domain $\Omega'$ such that 
\[
\{p\in \H: |p^{-1}\cdot q|_G\leq |h|, \ q\in \Oba\}\subset \Omega'.
\]
Applying the comparison principle, Theorem \ref{thm comp}, in $\Omega'$, we are led to
\[
u^K_{1, h}(p)-\omega_f(|h|_G)\leq u^K_1(p)\quad \text{in $\Omega'$,}
\]
which implies \eqref{equiconti1} with $n=2$. 

We can repeat the argument to show that \eqref{equiconti1} holds for all $n\geq 1$. Exchanging the roles of $u^K_n$ and $u^K_{n, h}$, we deduce that 
\[
|u^K_{n, h}-u^K_{n}|\leq \omega_f(|h|_G) \quad \text{in $\H$},
\]
which yields 
\[
|u_{n}(h\cdot p)-u_{n}(p)|\leq \omega_f(|h|_G) 
\]
for all for all $p\in (h^{-1}\cdot\Oba)\cap \Oba$ and $n\geq 1$. Due to the arbitrariness of $h$, we get the equi-continuity of $u_n$ in $\Oba$ with modulus $\omega_f$. It follows that 
$Q(f)=\inf_{n\geq 1} u_n$ is also continuous in $\Oba$ with the same modulus. 
The uniform convergence of $u_n$ to $Q(f)$ in $\Oba$ as $n\to \infty$ is an immediate consequence of Dini's theorem. 
\end{proof}

\begin{rmk}\label{rmk convergence2}
It is possible to use the same scheme to approximate the h-quasiconvex envelope $Q(f)$ of a function $f\in C(\Oba)$ that takes general boundary values. In this case, the uniqueness and existence of $u_n$ are guaranteed by Theorem \ref{thm comp2} and Theorem \ref{thm existence2}. The pointwise convergence of $u_n$ to $Q(f)$ can be shown in the same way as in the proof of Theorem \ref{thm scheme bdry}. It is also possible to obtain uniform convergence if $\ul{f}$ can be extended to an h-quasiconvex function in a neighborhood of $\Oba$. 
\end{rmk}

We conclude this section by providing a relaxed version of Theorem \ref{thm scheme bdry} with no assumptions on boundedness or convexity of $\Omega$.  
Based on Proposition \ref{prop sub}, for any $f\in USC(\Omega)$ that is bounded below, one can still obtain a sequence $u_n$ that converges pointwise to $Q(f)$ in $\Omega$ without even using the boundary value of $f$. In this general case, we can take $u_n$ to be the maximum subsolution of \eqref{iteration eq}. 
 
More precisely, let $u_0=f$ and, for $n=1, 2, \ldots$, let $u_n$ be the maximum subsolution of \eqref{iteration eq}, that is, 
\beq\label{iteration sub1}
u_n=v_n^\ast\quad \text{in $\Omega$},
\eeq
where, for any $p\in \Omega$, 
\beq\label{iteration sub2}
v_n(p)=\sup\{u(p): \ \text{$u\in USC(\Omega)$ is a subsolution of \eqref{hj eq} with $f=u_{n-1}$}\}.
\eeq
By Proposition \ref{prop sub}, $u_n$ is indeed a subsolution of \eqref{hj eq} with $f=u_{n-1}$.

\begin{thm}[Iterative scheme for envelope without boundary data]\label{thm scheme}
Let $\Omega$ be an h-convex domain in $\H$ and $f\in USC(\Omega)$. Assume that there exists an h-quasiconvex function $\ul{f}\in USC(\Omega)$ such that $f\geq \ul{f}$ in $\Omega$. 
Let $u_n\in USC(\Omega)$ be iteratively defined by \eqref{iteration sub1}--\eqref{iteration sub2}. Then $u_n\to Q(f)$ pointwise in $\Omega$ as $n\to \infty$, where $Q(f)\in USC(\Omega)$ is the h-quasiconvex envelope of $f$. 
\end{thm}
\begin{proof}
By Lemma \ref{lem sub}, we still have \eqref{monotone bdry}. 
It then follows from Remark \ref{rmk limits} and Proposition \ref{prop sub-stability} that $\lim_{n\to \infty} u_n\in USC(\Omega)$ and
\[
\lim_{n\to \infty} u_n\leq Q(f) \quad \text{in $\Omega$.}
\] 
We can also get the reversed inequality by showing that $Q(f)\leq u_n$ in $\Omega$ for all $n\geq 1$. Note that here we cannot prove it by the comparison principle, Theorem \ref{thm comp},  due to the possible unboundedness of $\Omega$ and loss of the boundary data. Instead, we use the maximality of $u_n$ among all subsolutions of \eqref{iteration eq}. 
\end{proof}

\section{H-convex hull}\label{sec:convex-hull}

In this section we study the h-convex hull of a given set in $\mathbb{H}$.
Using the definition of h-convex sets introduced in Definition \ref{def h-set}, we can define the h-convex hull of a set in the following natural way. 

\begin{defi}[H-convex hull]\label{def hull}
For a set $E \subset \H$ we denote by $\coh (E)$ the h-convex hull of $E$ defined to be the smallest h-convex set in $\H$ containing $E$, i.e., 
\[
\coh(E)=\bigcap \, \{D\subset \H: \text{$D$ is h-convex and satisfies $E\subset D$} \}.
\]
\end{defi}
Below we attempt to understand several basic properties of h-convex hulls. 

\subsection{Level set formulation}
We first establish connection between h-quasiconvex envelopes and h-convex hulls. Our general process of convexifying an open or closed set is an adaptation of the so-called level set method, which can be summarized as follows. 
\begin{enumerate}
\item For a given open (resp., closed) set $E\subset \H$, we take a function $f\in C(\H)$ such that 
\beq\label{appl1}
E= \{p\in \H: f(p)<0\} \quad (\text{resp., }  \ E=\{p\in \H: f(p)\leq 0\}).
\eeq
\item We construct the h-quasiconvex envelope $Q(f)$. 
\item The h-convex hull turns out to be the $0$-sublevel set of $Q(f)$, that is, 
\[
\coh(E)=\{p\in \H: Q(f)(p)<0\} \quad (\text{resp., }  \ \coh(E)=\{p\in \H: Q(f)(p)\leq 0\}).
\]
\end{enumerate}

Let us prove the result stated in the step (3) above under more precise assumptions. We first examine the case when $E$ is a bounded open or closed set in $\H$. Then we can take $\Omega=B_{R}(0)\supset \overline{\coh(E)}$ with $R>0$ large and use Theorem \ref{thm scheme bdry} to construct the h-quasiconvex envelope $Q(f)\in C(\Omega)$ for a defining function $f\in C(\Omega)$ of the set $E$.

\begin{thm}[Level set method for h-convex hull]\label{thm hull1}
Let $E\subset \H$ be a bounded open (resp., closed) set. Let $R>0$ large such that $\overline{\coh(E)}\subset B_R(0)$. 
Assume that $f\in C(\H)$ satisfies \eqref{appl1}. Assume also that there exists $K>0$ such that 
\beq\label{tech cond1}
\begin{cases}
f\leq K &\quad \text{in $B_R(0)$,}\\
f\equiv K  &\quad \text{in $\H\setminus B_{R}(0)$.} 
\end{cases}
\eeq
Let $Q(f)$ be the h-quasiconvex envelope of $f$. Then, $Q(f)\in C(\H)$ and $Q(f)$ satisfies 
\beq\label{appl3}
Q(f)\equiv K \quad \text{in $\H\setminus B_R(0)$}
\eeq
and
\beq\label{appl2}
\coh(E)=\{p\in \Omega: Q(f)(p)<0\} \quad (\text{resp., }  \ \coh(E)=\{p\in \Omega: Q(f)(p)\leq 0\}).
\eeq
\end{thm}
\begin{proof}
Let us only give a proof in the case when $E$ is a bounded open set. One can use the same argument to prove the result for a bounded closed set $E$.
In our current setting, we take $\Omega=B_R(0)$ and choose
\[
\ul{f}(p)=\min\{L(|p|_G-R), 0\}+K, \quad p\in \H
\]
with $L>0$ large so that all of the assumptions in Theorem \ref{thm scheme bdry} are satisfied for any $n>0$. We thus have $Q(f)\in C(\Oba)$ and \eqref{appl3}. 

Note that $\coh(E)\subset\{p\in \H: Q(f)(p)<0\}$
holds, since the right hand side is h-quasiconvex and contains $E$. In order to show \eqref{appl2}, it thus suffices to prove the reverse implication. To this end, we recall from Proposition \ref{prop metric-quasi} that 
\[
\psi_{\coh(E)}:=-\tilde{d}_H(\cdot, \H\setminus \coh(E))
\]
is an h-quasiconvex function in $\H$.
We adopt the same proof of Proposition \ref{prop lower} to construct an h-quasiconvex function $\ul{f}:=g\circ \psi_{\coh(E)}$ in $\Omega$ satisfying 
$f(p)\geq \ul{f}(p)$ for all $p\in \Omega$,
where $g$ is determined by the modulus of continuity $\omega_f$ of $f$, given by \eqref{g-fun}. It follows that
\[
Q(f)\geq Q(g\circ \psi_{\coh(E)})=g\circ \psi_{\coh(E)}.
\]
Since $g$ is actually strictly increasing, we are led to
\[
\begin{aligned}
\{p\in \H: Q(f)(p)<0\}&\subset \{p\in \H: g(\psi_{\coh(E)}(p))<0\}\\
&=\{p\in \H: \psi_{\coh(E)}(p)<0\}=\coh(E).
\end{aligned}
\]
The proof is now complete. 
\end{proof}

As an immediate consequence of the result above, if $E\subset \H$ is a bounded open set, its h-convex hull $\coh(E)$ is also bounded and open as well. 
Likewise, $\coh(E)$ is bounded and closed provided that $E\subset \H$ is bounded and closed. 

The h-convex hull $\coh(E)$ essentially does not depend on the choices of $\Omega$ and $f\in C(\H)$ as long as $\Omega$ is large enough to contain $\coh(E)$ and the $0$-sublevel set of $f$ agrees with $E$. 

Following a similar argument as in the proof of Theorem \ref{thm hull1}, one can construct the h-convex hull of a general, possibly unbounded open set $E\subset \H$. In this case, we study $\coh(E)$ in $\Omega=\H$ without imposing technical assumptions like \eqref{tech cond1}. We can use Theorem \ref{thm scheme} to get $Q(f)\in USC(\H)$.

\begin{thm}\label{thm hull2}
Let $E\subset \H$ be an open set.  
Assume that $f\in C(\H)$ satisfies 
\[
E= \{p\in \H: f(p)<0\}.
\]
Let $Q(f)$ be the h-quasiconvex envelope of $f$. Then $Q(f)\in USC(\H)$ and
\[
\coh(E)=\{p\in \H: Q(f)(p)<0\}. 
\]
\end{thm}
The proof is omitted, since it resembles that of Theorem \ref{thm hull1}.

\subsection{Quantitative inclusion principle}

The definition of h-convex hulls, Definition \ref{def hull}, immediately yields the inclusion principle: $\coh(D)\subset \coh(E)$ if $D\subset E$.
Using the formulation via h-quasiconvex envelopes, we can quantify this result for bounded open or closed sets. 
\begin{thm}[Quantitative inclusion principle]\label{thm sep}
Let $D, E$ be two bounded open (resp., closed) sets in $\H$. If $D\subset E$, then $\coh(D)\subset \coh(E)$ and  \eqref{dist sep} holds.
\end{thm}
In order to prove this result, let us introduce the following sup-convolution for a bounded continuous function $u$ with $\delta>0$: 
\beq\label{sup-convolution}
u^\delta(p)=\sup_{\tilde{d}_H(p, q)< \delta} u(q),\quad p\in \H.
\eeq
Due to the continuity of $u$, we also have
\[
u^\delta(p)=\max_{\tilde{d}_H(p, q)\leq  \delta} u(q), \quad p\in \H.
\]
It turns out that this sup-convolution also preserves h-quasiconvexity. 
\begin{lem}[H-quasiconvexity preserving by sup-convolution]\label{lem sup}
Assume that $u$ is a bounded continuous function in $\H$. Let $u^\delta\in C(\H)$ be given by \eqref{sup-convolution}. Then $u^\delta$ is h-quasiconvex in $\H$ if $u$ is h-quasiconvex in $\H$. 
\end{lem}
\begin{proof}
Note that $u^\delta$ is bounded and continuous in $\H$, since $u$ is bounded and continuous. We approximate $u^\delta$ by 
\beq\label{sup approx}
u^\delta_\beta(p)=\sup_{q\in \H}\left\{u(q)-\frac{\tilde{d}_H(p, q)^\beta}{\delta^\beta}\right\}
\eeq
with $\beta>0$ large. It is easily seen that $u^\delta_\beta$ is also bounded and continuous in $\H$. For any $p_0\in \H$, we can find $q_{\delta, \beta}\in \H$ such that 
\beq\label{sep eq4}
u^\delta_\beta(p_0)=u(q_{\delta, \beta})-\frac{\tilde{d}_H(p_0, q_{\delta, \beta})^\beta}{\delta^\beta},
\eeq
which yields 
\[
 \frac{\tilde{d}_H(p_0, q_{\delta, \beta})^\beta}{\delta^\beta}\leq u(q_{\delta, \beta})-u^\delta_\beta(p_0).
 \]
Due to the boundedness of $u$, sending $\beta\to \infty$, we obtain 
\[
\limsup_{\beta\to \infty}\tilde{d}_H(p_0, q_{\delta, \beta})\leq \delta. 
\]
We thus can take a subsequence of $q_{\delta, \beta}$ converging to a point $q_\delta\in \H$, which satisfies $\tilde{d}_H(p_0, q_{\delta})\leq \delta$. It follows that 
\beq\label{sep eq1}
\limsup_{\beta\to \infty} u^\delta_\beta(p_0)\leq u(q_\delta)\leq \max_{\tilde{d}_H(p_0, q)\leq  \delta} u(q)=u^\delta(p_0).
\eeq
On the other hand, noticing that
\[
u^\delta_\beta(p_0)\geq \sup\left\{u(q)-\frac{\tilde{d}_H(p_0, q)^\beta}{\delta^\beta}: \tilde{d}_H(p_0, q)<\delta\right\},
\]
we deduce that
\beq\label{sep eq2}
\sup_{\beta\geq 1} u^\delta_\beta(p_0)\geq \sup\left\{u(q): \tilde{d}_H(p, q)<\delta\right\}=u^\delta(p_0).
\eeq
Combining \eqref{sep eq1} and \eqref{sep eq2} as well as the arbitrariness of $p_0$, we are led to 
\beq\label{sep eq3}
u^\delta(p)=\sup_{\beta\geq 1} u^\delta_\beta(p)\quad\text{for all $p\in \H$.}
\eeq

Suppose that there exist $\varphi\in C^1(\H)$ and $p_0\in \H$ such that $u^\delta_\beta-\varphi$ attains a local maximum at $p_0$. Then
\beq\label{sep eq5}
(p, q)\mapsto u(q)-\varphi(p)-\frac{\tilde{d}_H(p, q)^\beta}{\delta^\beta}
\eeq
attains a maximum at $(p_0, q_{\delta, \beta})$, where $q_{\delta, \beta}\in \H$ is the point satisfying \eqref{sep eq4}. The maximality of \eqref{sep eq5} implies that $u-\psi$ attains a maximum at $q_{\delta, \beta}$, where $\psi\in C^1(\H)$ is given by
\[
\psi(q)=\varphi(p_0)+\frac{\tilde{d}_H(p_0, q)^\beta}{\delta^\beta}.
\]
Since $u$ is h-quasiconvex, we have 
\beq\label{sep eq6}
\sup\left\{\la \nabla_H \psi(q_{\delta, \beta}),  (q_{\delta, \beta}^{-1}\cdot \eta)_h\ra:\ \eta\in S_{q_{\delta, \beta}}(u) \right\}\leq 0.
\eeq
Note that for any $\xi \in S_{p_0} (u^\delta_\beta)$, we can choose $\eta=q_{\delta, \beta}\cdot p_0^{-1}\cdot \xi$ so that $\eta\in S_{q_{\delta, \beta}} (u)$. Indeed, we can easily verify that $\eta\in \H_{q_{\delta, \beta}}$ and by \eqref{sup approx} and \eqref{sep eq4} we obtain
\[
u(\eta)-\frac{\tilde{d}_H(\xi, \eta)^\beta}{\delta^\beta}\leq u^\delta_{\beta}(\xi)<u^\delta_\beta(p_0)=u(q_{\delta, \beta})-\frac{\tilde{d}_H(p_0, q_{\delta, \beta})^\beta}{\delta^\beta}
\]
which reduces to  $u(\eta)<u^\delta_\beta(p_0)$
due to the fact that
\[
\tilde{d}_H(\xi, \eta)=\tilde{d}_H(p_0, q_{\delta, \beta}).
\]
Hence, we have $\eta\in S_{q_{\delta, \beta}} (u)$. 

Since $p_0^{-1}\cdot \xi=q_{\delta, \beta}^{-1}\cdot \eta$ and $\nabla_H\varphi(p_0)=\nabla_H \psi(q_{\delta, \beta})$ (similar to the calculations for \eqref{cp new5}), \eqref{sep eq6} can be rewritten as 
\[
\sup\left\{\la \nabla_H \varphi(p_0),  ({p_0}^{-1}\cdot \xi)_h\ra:\ \xi \in S_{p_0} (u^\delta_\beta) \right\}\leq 0.
\]
By Theorem \ref{thm char}, we see that $u^\delta_\beta$ is h-quasiconvex in $\H$. Thanks to \eqref{sep eq3}, we conclude the proof of the h-quasiconvexity of $u^\delta$ by applying the result in Remark \ref{rmk supremum}.
\end{proof}

\begin{proof}[Proof of Theorem \ref{thm sep}]
Since $E\subset \coh(E)$, we get $D\subset \coh(E)$. Due to the h-convexity of $\coh(E)$, by Definition \ref{def hull} we obtain $\coh(D)\subset \coh(E)$. We next prove \eqref{dist sep}, assuming that $E, D$ are open. The case for closed sets can be similarly treated. Take $\Omega=B_{R}(0)$ and $f\in C(\H)$ as in Theorem \ref{thm hull1}.  In this case, $Q(f)\in C(\H)$ and \eqref{appl2} hold. Since $Q(f)$ is constant outside $B_{R}(0)$, we can also obtain the boundedness of $Q(f)$. 

Adopting Lemma \ref{lem sup}, we see that the sup-convolution $Q(f)^\delta$ of $Q(f)$ is h-quasiconvex in $\H$ for any $\delta>0$.  
It follows that
\[
\coh(E)_\delta:=\{p\in \H: Q(f)^\delta(p)<0\}=\{p\in \H: \tilde{B}_p(\delta)\subset \coh(E)\}
\]
is h-convex, where we recall that $\tilde{B}_p(\delta)$ denotes the right-invariant metric ball centered at $p$ with radius $\delta$. 
By taking $\delta=\inf_{p\in \H\setminus E}\tilde{d}_H(p, D)$, we have 
\[
D\subset \{p\in \H: \tilde{B}_p(\delta)\subset E\}\subset \coh(E)_\delta,
\]
which yields $\coh(D)\subset \coh(E)_\delta$ by the h-convexity of $\coh(E)_\delta$. This gives \eqref{dist sep} immediately.
\end{proof}

\subsection{Continuity under star-shapedness}\label{sec:cont}
To conclude this work, we briefly mention the continuity of $\coh(E)$ with respect to $E$. Let us recall that for any two sets $D,E \subset \H$, the distance $d_H(D,E)$ is defined as 
\[
d_H(D,E)=\max \left\{\sup_{p \in D} d_H(p,E),\ \sup_{p \in E} d_H(p,D)\right\}.
\]
We are interested in the following question: for a bounded open or closed set $E\subset \H$ and a sequence of sets $E_j\subset \H$, is it true that 
\beq\label{hull conti}
d_H(\coh(E), \coh(E_j))\to 0\quad \text{if}\quad d_H(E, E_j)\to 0
\eeq
holds?

In contrast to the Euclidean case, where the answer is affirmative, the situation in the Heisenberg group is less straightforward. In general, we cannot expect \eqref{hull conti} to hold, as indicated by the following example.

\begin{example}\label{ex2}
We slightly change the set $E$ in Example \ref{ex1} by taking 
\[
E =(\pi(0, r) \times (-\delta, 0)) \cup (\pi(0, r) \times (t,t+ \delta))
\]
with $r, t, \delta>0$. In other words, $E$ is set to be the union of two circular cylinders of height $\delta>0$. We may use the same argument as in Example \ref{ex1} to show that $E$ is h-convex if $r^2\leq 2t$ but not h-convex if $r^2>2t$. 

Consider the critical case when $r^2=2t$. Since $E$ is h-convex in this case, we have $\coh(E)=E$. On the other hand, it is not difficult to see that any $\vep$-neighborhood of $E$ (with respect to $d_H$ or other equivalent metrics), denoted by $\N_\vep(E)$, is not h-convex. Its h-convex hull $\coh(\N_\vep(E))$ at least contains horizontal segments joining the upper and lower cylinders. Hence, we can get $c>0$ such that
\[
d_H(\coh(E), \coh(\N_\vep(E))>c
\]
for all $\vep>0$. This shows that \eqref{hull conti} fails to hold in general. 
\end{example}

We have seen that the h-convex hull $\coh(E)$ is not stable with respect to perturbations of the set $E$. It is however worth emphasizing that the h-quasiconvex envelope is stable, as shown in Proposition \ref{prop stability}. The main reason for such a discrepancy is that the quasi-convexification process may fatten the level sets of functions, that is, a level set of $Q(f)$ may contain interior points while the level set of $f$ at the same level initially does not. See for example \cite{BSS, So3, Gbook} on the fattening phenomenon for the level-set formulation of geometric evolution equations in different contexts. 

We need additional assumptions to guarantee \eqref{hull conti}. One sufficient condition is the following type of star-shapedness of the set $E$. 

\begin{defi}[Strict star-shapedness]
We say that a connected set $E\subset \H$ is strictly star-shaped with respect to the group identity $0\in H$ if for any $0\leq \lambda<1$ we have $\delta_\lambda (E)\subset E$ and 
\beq\label{strict star eq}
\inf_{x\in \delta_\lambda(E)} d_H(x, \H\setminus E)>0.
\eeq 
We say $E$ is strictly star-shaped if there exists a point $p_0\in \H$ such that 
the left translation $p_0^{-1} \cdot E$ of $E$ is strictly star-shaped with respect to $0$.  
\end{defi}
It is a stronger version than the star-shapedness property studied in \cite{DF, DGS}, where all conditions except \eqref{strict star eq} are required. Note that h-convex set can be disconnected and does not imply strict or weak star-shapedness, as observed in Example \ref{ex1}. Below, we present three examples of connected sets to compare these notions. We refer the interested reader to \cite{DF} for more examples of star-shaped sets.
\begin{example}
Let us take
\begin{align*}
 E_1 &= \textstyle\bigcup_{\lambda \in [0,1]} \delta_{\lambda}((1,1,1)) \cup \delta_{\lambda}((-1,-1,1)) ,\\ 
E_2 & =(\pi(0,1)\times \{0\})\cup (\pi(0,1)\times \{t\})\cup (\{(0,0)\}\times (0,1)),\\ 
E_3 &= \{(x, 0, 0): x\in (-1, 1)\},
\end{align*}
where $\pi(0, 1)$ denotes the unit disk centered at the origin as defined in \eqref{h-pi}.  

All of these sets are bounded and connected. The set $E_1$ is strictly star-shaped with respect to the group identity but not h-convex; it is easily seen that for any $\delta_s(E_1)\subset E_1$ for all $0\leq s<1$ but the horizontal segment between $(1, 1, 1)$ and $(-1, -1, 1)$ does not stay in $E_1$. In view of Example \ref{ex1}, we see that $E_2$ is h-convex. But it is not star-shaped, since for any $p_0\in E_2$, there exists $p\in E_2$ such that the curve $\{p_0\cdot \delta_\lambda(p_0^{-1}\cdot p): \lambda\in [0, 1)\}$
contains points outside $E_2$. 
The horizontal segment $E_3$ is obviously h-convex and star-shaped with respect to every point in $E_3$. It is however not star-shaped with respect to points in $\H\setminus E_3$. It is not strictly star-shaped in $\H$ either. Indeed, for any $p_0\in E_3$, we can find $\lambda<1$ close to $1$ such that $0\in \delta_\lambda(p_0^{-1}\cdot E_3)$ and $d_H(0, \H\setminus (p_0^{-1}\cdot E_3))=0$; in other words, $p_0^{-1}\cdot E_3$ is not strictly star-shaped with respect to $0$. 
\end{example}

 Assuming strict star-shapedness, we prove stability of h-convex hull.
\begin{prop}[Stability of h-convex hulls of star-shaped sets]\label{prop stable}
Let $E\subset \H$ be a bounded open (resp., closed) set that is strictly star-shaped. Let $E_j$ be a sequence of bounded open (resp. closed) sets in $\H$ such that $d_H(E, E_j)\to 0$ as $j\to \infty$. Then there holds 
\beq\label{set stability}
d_H(\coh(E), \coh(E_j))\to 0\quad \text{as $j\to \infty$.}
\eeq
\end{prop}
\begin{proof}
By definition of h-convex hull and $\delta_\lambda (p)\in \H_{\delta_\lambda (q)}$ provided $p\in \mathbb{H}_q$, it is not difficult to verify that 
\beq\label{starshape eq1}
\coh(\delta_\lambda E)=\delta_\lambda (\coh(E))
\eeq
for any $\lambda> 0$. 

Owing to the strict star-shapedness condition on $E$, for each fixed $0<\lambda<1$ we have 
$\delta_\lambda (E)\subset E_j$
when $j\geq 1$ is sufficiently large. It then follows that $\coh(\delta_\lambda (E))\subset \coh( E_j)$.
By \eqref{starshape eq1}, we get $\delta_\lambda (\coh(E))\subset \coh(E_j)$ for all $j\geq 1$ large. 

Since the strict star-shapedness also implies $E\subset \delta_{1/\lambda} (E)$ and 
\[
d_H(E, \delta_{1/\lambda} (E))>0
\]
for any $0<\lambda<1$, we can use a symmetric argument to show that 
$\coh(E_j)\subset \delta_{1/\lambda} (\coh(E))$ 
for $j\geq 1$ large. Hence, $\delta_{\lambda} (\coh(E))\subset \coh(E_j)\subset \delta_{1/\lambda} (\coh(E))$ and $\delta_{\lambda} (\coh(E))\subset \coh(E)\subset \delta_{1/\lambda} (\coh(E))$ imply that
\[
d_H(\coh(E), \coh(E_j)) \leq d_H\left(\delta_\lambda (\coh(E)), \delta_{1/\lambda}(\coh(E))\right)
\]
for $j\geq 1$ large. Letting $j\to \infty$ and then $\lambda\to 1$ yields \eqref{set stability}.
\end{proof}
The strict star-shapedness is only a sufficient condition to guarantee the stability of h-convex hull. It would be interesting to find a sufficient and necessary condition for this property. 

\subsection*{Acknowledgments}
The authors are grateful to the anonymous referees for valuable comments, especially to one of the referees for pointing out useful references that help us prove Lemma \ref{rmk convex boundary} and improve the result in Theorem \ref{thm comp2}.
The work of the second author was supported by JSPS Grants-in-Aid for Scientific Research (No.~19K03574, No.~22K03396) and by funding (Grant No. 205004) from Fukuoka University. The work of the third author was supported by JSPS Grant-in-Aid for Research Activity Start-up (No.~20K22315) and JSPS 	
Grant-in-Aid for Early-Career Scientists (No.~22K13947).

\bibliographystyle{abbrv}%

\end{document}